\documentclass[11pt]{article}
\usepackage[margin=1in]{geometry}
\usepackage{amsmath,amssymb,amsthm,mathtools}
\usepackage{mathrsfs}
\usepackage{booktabs,array}
\usepackage{microtype}
\usepackage[hidelinks]{hyperref}
\usepackage[T1]{fontenc}
\usepackage{lmodern}
\allowdisplaybreaks

\newtheorem{theorem}{Theorem}[section]
\newtheorem{lemma}[theorem]{Lemma}
\newtheorem{proposition}[theorem]{Proposition}
\newtheorem{corollary}[theorem]{Corollary}
\theoremstyle{definition}
\newtheorem{definition}[theorem]{Definition}

\newtheorem{example}[theorem]{Example}

\newcommand{\Energy}{\mathcal E}

\title{Collision Positivity for Three-Variable Symmetric Monomial Inequalities:\\A Complete Three-Variable Criterion}
\author{Jian Sun \\jiansun@uchicago.edu}
\date{\today}

\begin{document}
\maketitle

\begin{abstract}
Let \(\lambda\succ\gamma\succ\mu\) be equal-degree exponent partitions with at most three parts, and let
\begin{equation*}
P_{\lambda,\gamma,\mu}
=
J_\lambda+J_\mu-2J_\gamma,
\end{equation*}
where \(J_\nu\) denotes the symmetric monomial orbit sum associated with \(\nu\).  We prove a necessary and sufficient collision criterion for the positivity of this three-point majorization difference in three variables.  For nonnegative integer exponent partitions,
\begin{equation*}
P_{\lambda,\gamma,\mu}(x,y,z)\ge0
\qquad(x,y,z>0)
\end{equation*}
if and only if
\begin{equation*}
P_{\lambda,\gamma,\mu}(t,1,1)\ge0
\qquad(t>0).
\end{equation*}
Thus the positivity of a genuinely three-variable symmetric polynomial of this form is completely determined by its one-variable restriction to the locus where two variables coincide.

The theorem gives a uniform and effectively checkable criterion for an infinite class of symmetric polynomial inequalities.  For a fixed integer chain, the global three-variable problem is reduced to a single univariate polynomial inequality, which can often be verified exactly by factorization, Sturm's theorem, or other one-variable methods.  The criterion extends well beyond the classical Schur family, produces explicit inequalities outside the Schur pattern, and yields genuine strengthenings and refinements of Schur's inequality.

The proof is organized by the geometry of a three-part dominance chain.  After a common translation, exact finite-difference identities for the collision restriction give a finite descent to minimal collision-positive states.  These states are divided according to the relative positions of two finite supports into crossing, seam, contact, and contained configurations.  Crossing and seam configurations are controlled by monotonicity of the associated support ratios together with ordered-minor and total-positivity arguments.  The boundary seam in which one of the two finite supports degenerates is handled by a direct three-channel decomposition and an elementary multiplicative-kernel inequality.  The remaining contact and contained configurations are closed by reciprocal complement and a maximal-depth minimal-counterexample argument; in particular, the proof does not assume that reciprocal complement preserves minimality.

We also give the exact two-variable criterion, extend the theorem to rational exponent vectors by denominator clearing, exhibit explicit low-degree inequalities beyond Schur, establish a rigidity property of the classical Schur chain, and derive a strict refinement of the cubic Schur inequality through a midpoint incomparable with the classical one.  We have also investigated analogous questions for \(n\ge4\) variables; the higher-dimensional results require a different development and will be presented separately.

\end{abstract}

\section{Introduction}

Muirhead's inequality compares symmetric monomial orbit sums whose exponent vectors are ordered by majorization.  Thus an equal-degree chain
\begin{equation*}
\lambda\succ\gamma\succ\mu
\end{equation*}
gives two nonnegative differences,
\begin{equation*}
J_\lambda-J_\gamma\ge0,
\qquad
J_\gamma-J_\mu\ge0.
\end{equation*}
The question studied in this paper is whether the first difference is at least as large as the second:
\begin{equation*}
J_\lambda-J_\gamma
\ge
J_\gamma-J_\mu.
\end{equation*}
Equivalently, we ask when
\begin{equation*}
P_{\lambda,\gamma,\mu}
=
J_\lambda+J_\mu-2J_\gamma
\end{equation*}
is nonnegative.

Schur's inequality is the classical example of such a three-orbit second difference, but it corresponds to a very special pattern of exponent vectors.  Section~2 makes this distinction explicit and gives a degree-nine example outside the Schur family.  It also recalls Timofte's half-degree principle for general symmetric polynomials.  In three variables that general principle ceases to reduce dimension once the degree is at least six.  The collision theorem proved here gives a much stronger reduction for this special three-orbit defect: positivity on the two-equal-variable collision locus is sufficient in every degree.

For a symmetric polynomial, the \emph{collision locus} is the set on which two variables coincide.  In three variables it is the union of the three equality sets \(x=y\), \(y=z\), and \(z=x\).  By symmetry and homogeneity it is enough to use the normalized restriction
\begin{equation*}
P_{\lambda,\gamma,\mu}(t,1,1),
\qquad t>0.
\end{equation*}
Our main theorem shows that nonnegativity of this one-variable collision restriction is not only necessary but also sufficient for nonnegativity on the whole positive three-variable cone.  Sections~\ref{sec:global-roadmap} and \ref{sec:one-sided-completion} record the sector decomposition and the final reciprocal completion.  We have also investigated the corresponding questions for \(n\ge4\) variables.  Since the higher-dimensional arguments have a substantially different structure, those results will be presented in a separate paper.

The proof is organized by the canonical geometry of a three-part dominance chain. After a common translation we write
\begin{equation*}
\lambda=(p,q,0),
\end{equation*}
and parameterize the moving middle partition by two lattice variables. The collision quotient has exact negative mixed second difference on this lattice, while the lower endpoint direction is separable. This yields a finite reduction to minimal collision-positive states. The resulting minimal states are divided by a single support parameter into crossing, seam, contact, and contained geometries.

The proof architecture separates four global regions.  In the preferred exterior orientation, crossing and the ordinary seam are controlled by monotonicity of two finite-support ratios and an ordered-minor uncrossing argument.  The double-exterior seam with \(V\in\{0,1\}\) is handled by a three-channel multiplicative-kernel argument.  Every contact and contained state, including the boundary slices, is then completed by a reciprocal maximal-depth argument with bounded transfer depth \(D=a+c\).

Sections~5--15 give this architecture in logical order.  A compact ordered-minor total-positivity lemma supplies the crossing argument, while the boundary seam is handled directly.  Section~\ref{sec:one-sided-completion} closes the remaining negative-\(r\) cases by the reciprocal maximal-depth argument.  An additional direct calculation is retained only as an independent verification and is not a logical dependency of the theorem.

\section{Classical Background}

This section recalls the classical tools needed later.  We begin with majorization and Muirhead's inequality, then review Schur's inequality and compare the present problem with the general half-degree principle.

\subsection{Majorization and Muirhead's inequality}

Majorization is the natural order behind symmetric monomial inequalities.  Let
\begin{equation*}
a=(a_1,\ldots,a_n),
\qquad
b=(b_1,\ldots,b_n)
\end{equation*}
be vectors with the same total sum, arranged in decreasing order.  We say that \(a\) majorizes \(b\), and write \(a\succcurlyeq b\), if
\begin{equation*}
\sum_{i=1}^k a_i\ge\sum_{i=1}^k b_i
\qquad(1\le k<n).
\end{equation*}
When the two vectors are different, we write \(a\succ b\).

For a nonnegative exponent partition \(\eta=(\eta_1,\ldots,\eta_n)\), define the labeled symmetric orbit sum by
\begin{equation*}
J_\eta(x_1,\ldots,x_n)
=\sum_{\sigma\in S_n}\prod_{i=1}^n x_i^{\eta_{\sigma(i)}}.
\end{equation*}
All \(n!\) labeled permutations are included, so repeated exponent values occur with their natural multiplicities.

\begin{theorem}[Muirhead's inequality]\label{thm:muirhead}
Let \(\eta\) and \(\xi\) be nonnegative exponent partitions of the same total degree.  If
\begin{equation*}
\eta\succcurlyeq\xi,
\end{equation*}
then
\begin{equation*}
J_\eta(x_1,\ldots,x_n)
\ge
J_\xi(x_1,\ldots,x_n)
\end{equation*}
for all \(x_1,\ldots,x_n\ge0\).
\end{theorem}

This is the classical Muirhead inequality \cite{Muirhead1902,HLP1952,MOA2011}.  In particular, if
\begin{equation*}
\lambda\succ\gamma\succ\mu,
\end{equation*}
then
\begin{equation*}
J_\lambda\ge J_\gamma\ge J_\mu.
\end{equation*}
Muirhead's theorem determines the signs of the two successive gaps, but it does not compare their sizes.  The problem of this paper is precisely to determine when
\begin{equation*}
J_\lambda-J_\gamma
\ge
J_\gamma-J_\mu.
\end{equation*}

\subsection{Schur's inequality}

Schur's inequality gives the fundamental classical example \cite{Schur1923}.

\begin{theorem}[Schur's inequality]\label{thm:schur}
For every real \(r\ge0\) and all \(x,y,z\ge0\),
\begin{equation*}
\sum_{\mathrm{cyc}}x^r(x-y)(x-z)\ge0.
\end{equation*}
Equivalently, under the labeled-orbit convention,
\begin{equation*}
J_{(r+2,0,0)}(x,y,z)
+J_{(r,1,1)}(x,y,z)
\ge
2J_{(r+1,1,0)}(x,y,z).
\end{equation*}
\end{theorem}

For completeness, we recall the short proof.

\begin{proof}
By symmetry assume \(x\ge y\ge z\ge0\).  Then
\begin{align*}
\sum_{\mathrm{cyc}}x^r(x-y)(x-z)
={}&x^r(x-y)(x-z)-y^r(x-y)(y-z)\\
&+z^r(x-z)(y-z).
\end{align*}
The first two terms equal
\begin{equation*}
(x-y)\bigl(x^r(x-z)-y^r(y-z)\bigr),
\end{equation*}
and
\begin{equation*}
x^r(x-z)-y^r(y-z)
=(x^r-y^r)(x-z)+y^r(x-y)\ge0.
\end{equation*}
The last term is also nonnegative.  This proves Schur's inequality.  Expanding the orbit sums gives the equivalent orbit form.
\end{proof}

For \(r\ge1\), the displayed exponent vectors are already partitions and form the majorization chain
\begin{equation*}
(r+2,0,0)
\succ
(r+1,1,0)
\succ
(r,1,1).
\end{equation*}
For \(0\le r<1\), the orbit identity remains valid and the final exponent vector is interpreted after decreasing rearrangement.  Thus Schur compares two successive Muirhead gaps along this one-parameter family, with the usual rearrangement convention when necessary.  It is natural to ask whether all inequalities of the form
\begin{equation*}
J_\lambda+J_\mu\ge2J_\gamma
\end{equation*}
are consequences of this family.  They are not.

\subsection{A degree-nine inequality outside the Schur family}

Consider
\begin{equation*}
(9,0,0)\succ(7,1,1)\succ(3,3,3).
\end{equation*}
For the same endpoints \((9,0,0)\) and \((3,3,3)\), Schur's family has middle partition \((6,3,0)\).  The two middle partitions are incomparable because
\begin{equation*}
7>6,
\qquad
7+1<6+3.
\end{equation*}
Consequently the following inequality is not obtained from Schur by a Muirhead comparison of the middle terms.

\begin{theorem}[A degree-nine inequality]\label{thm:degree-nine}
For all \(a,b,c>0\),
\begin{equation*}
J_{(9,0,0)}(a,b,c)+J_{(3,3,3)}(a,b,c)
\ge
2J_{(7,1,1)}(a,b,c).
\end{equation*}
\end{theorem}

\begin{proof}
With the labeled-orbit convention,
\begin{equation*}
J_{(9,0,0)}=2(a^9+b^9+c^9),
\qquad
J_{(3,3,3)}=6(abc)^3,
\end{equation*}
and
\begin{equation*}
J_{(7,1,1)}=2abc(a^6+b^6+c^6).
\end{equation*}
Thus it is enough to prove
\begin{equation*}
a^9+b^9+c^9+3(abc)^3
\ge
2abc(a^6+b^6+c^6).
\end{equation*}
Divide by \((abc)^3\) and set
\begin{equation*}
\xi=\frac{a^2}{bc},
\qquad
\eta=\frac{b^2}{ca},
\qquad
\zeta=\frac{c^2}{ab}.
\end{equation*}
Then
\begin{equation*}
\xi\eta\zeta=1,
\end{equation*}
and the desired inequality becomes
\begin{equation}\label{eq:degree9-scaled}
\xi^3+\eta^3+\zeta^3+3
\ge
2(\xi^2+\eta^2+\zeta^2).
\end{equation}
By symmetry suppose that \(\xi=\min\{\xi,\eta,\zeta\}\).  Since \(\xi\eta\zeta=1\), we have \(\eta\zeta\ge1\).  Put
\begin{equation*}
t=\sqrt{\eta\zeta}\ge1,
\qquad
q=\sqrt\eta,
\qquad
r=\sqrt\zeta,
\end{equation*}
so that \(qr=t\).  Define
\begin{equation*}
f(\xi,\eta,\zeta)
=\xi^3+\eta^3+\zeta^3
-2(\xi^2+\eta^2+\zeta^2).
\end{equation*}
Replacing \(\eta,\zeta\) by their geometric mean cannot increase this expression, because
\begin{align*}
f(\xi,\eta,\zeta)-f(\xi,t,t)
&=(q^3-r^3)^2-2(q^2-r^2)^2\\
&=(q-r)^2\Bigl((q^2+qr+r^2)^2-2(q+r)^2\Bigr)\ge0.
\end{align*}
Indeed, using \(qr\ge1\),
\begin{align*}
(q^2+qr+r^2)^2
&\ge q^4+r^4+2q^2r^2+2qr(q^2+r^2)\\
&\ge4q^2r^2+2(q^2+r^2)\\
&\ge2(q+r)^2.
\end{align*}
Since \(\xi t^2=1\), we have \(\xi=t^{-2}\).  It remains to check
\begin{equation*}
f(t^{-2},t,t)+3\ge0
\qquad(t\ge1).
\end{equation*}
After multiplication by \(t^6>0\), this is
\begin{equation*}
2t^9-4t^8+3t^6-2t^2+1\ge0.
\end{equation*}
Finally,
\begin{align*}
2t^9-4t^8+3t^6-2t^2+1
=(t-1)^2\Bigl(&t^3(t^2-1)^2+t(t^3-1)^2\\
&+t^2(t^2-t+1)+t+1\Bigr),
\end{align*}
which is nonnegative for \(t\ge1\).  This proves \eqref{eq:degree9-scaled} and hence the theorem.
\end{proof}

This example shows concretely that the three-orbit problem extends beyond the classical Schur family.

\subsection{Comparison with the half-degree principle}

There is a general reduction theorem for symmetric polynomial inequalities.  The following form is due to Timofte, with a later elementary proof by Riener \cite{Timofte2003,Riener2012}.

\begin{theorem}[Half-degree principle]\label{thm:half-degree}
Let \(f\) be a real symmetric polynomial of degree \(d\) in \(n\) variables.  Put
\begin{equation*}
k=\max\left\{2,\left\lfloor\frac d2\right\rfloor\right\}.
\end{equation*}
Then \(f\) is nonnegative on \(\mathbb R_{\ge0}^n\) if and only if it is nonnegative at every point having at most \(k\) distinct coordinates.
\end{theorem}

For \(n=3\), this is a genuine reduction to the two-equal-variable boundary only while \(k\le2\).  Once \(d\ge6\), one has \(k\ge3\), and every point of \(\mathbb R_{\ge0}^3\) already has at most three distinct coordinates.  The general half-degree principle then gives no dimensional reduction.  In particular, for the degree-nine inequality above it asks us to test the full three-variable domain.

The theorem proved here is different in character.  For the special second differences considered here, nonnegativity on the two-equal-variable boundary is sufficient in every degree.  This degree-independent reduction is one of the main reasons for studying the collision condition.

For the remainder of the paper, let
\begin{equation*}
\lambda\succ\gamma\succ\mu
\end{equation*}
be an equal-degree majorization chain and define
\begin{equation*}
P_{\lambda,\gamma,\mu}
=J_\lambda+J_\mu-2J_\gamma.
\end{equation*}

\section{The Two-Variable Collision Criterion}

The two-variable case is governed completely by the behavior at the collision \(x=y\).  We first introduce collision coordinates and the associated quadratic quantity, then state the exact criterion.

\subsection{Collision coordinates and quadratic data}

We first examine the problem in two variables.  The set where the two positive variables coincide is the diagonal
\begin{equation*}
x=y.
\end{equation*}
We call this diagonal the \emph{collision locus}.  Since the orbit sums are homogeneous and symmetric, every positive pair can be written uniquely in the form
\begin{equation*}
x=ge^t,
\qquad
y=ge^{-t},
\qquad
g>0,
\end{equation*}
with \(t\in\mathbb R\).  The collision \(x=y\) corresponds exactly to \(t=0\).  After the common homogeneous factor is removed, the two-variable defect is therefore an even function of one real parameter.  Its constant and linear terms vanish at the collision, so the first possible obstruction is quadratic.

For an equal-degree chain define the quadratic quantity
\begin{equation*}
\Energy(\lambda,\gamma,\mu)
=\frac12\left(
\|\lambda\|_2^2+\|\mu\|_2^2-2\|\gamma\|_2^2
\right).
\end{equation*}
The next theorem shows that in two variables this single number is a complete criterion.

\subsection{The exact criterion}

\begin{theorem}[Two-variable collision criterion]\label{thm:two-variable}
Let
\begin{equation*}
\lambda\succ\gamma\succ\mu
\end{equation*}
be equal-degree nonnegative two-part partitions.  Real exponents are allowed.  Then
\begin{equation*}
P_{\lambda,\gamma,\mu}(x,y)\ge0
\qquad(x,y>0)
\end{equation*}
if and only if
\begin{equation*}
\Energy(\lambda,\gamma,\mu)\ge0.
\end{equation*}
\end{theorem}

\begin{proof}
Let the common degree be \(N\), and write
\begin{equation*}
x=ge^t,
\qquad
y=ge^{-t}.
\end{equation*}
For a two-part partition \(\eta\), let
\begin{equation*}
\delta_\eta=\eta_1-\eta_2\ge0.
\end{equation*}
Then
\begin{equation*}
J_\eta(x,y)=2g^N\cosh(\delta_\eta t).
\end{equation*}
Hence the desired inequality is equivalent to
\begin{equation*}
\cosh(\delta_\lambda t)+\cosh(\delta_\mu t)
\ge2\cosh(\delta_\gamma t)
\qquad(t\in\mathbb R).
\end{equation*}
The quadratic term at \(t=0\) shows that this requires
\begin{equation*}
\delta_\lambda^2+\delta_\mu^2\ge2\delta_\gamma^2,
\end{equation*}
which is equivalent to \(\Energy(\lambda,\gamma,\mu)\ge0\).

Conversely, fix \(t\in\mathbb R\) and define
\begin{equation*}
f_t(z)=\cosh(t\sqrt z),
\qquad z\ge0.
\end{equation*}
Its power-series expansion is
\begin{equation*}
f_t(z)
=\sum_{k=0}^{\infty}\frac{t^{2k}z^k}{(2k)!}.
\end{equation*}
Termwise differentiation gives
\begin{equation*}
f_t'(z)
=\sum_{k=1}^{\infty}
\frac{k\,t^{2k}z^{k-1}}{(2k)!}
\ge0,
\end{equation*}
and
\begin{equation*}
f_t''(z)
=\sum_{k=2}^{\infty}
\frac{k(k-1)t^{2k}z^{k-2}}{(2k)!}
\ge0
\qquad(z\ge0).
\end{equation*}
Hence \(f_t\) is increasing and convex on \([0,\infty)\).  Therefore
\begin{align*}
\frac{\cosh(\delta_\lambda t)+\cosh(\delta_\mu t)}2
&=\frac{f_t(\delta_\lambda^2)+f_t(\delta_\mu^2)}2\\
&\ge
f_t\!\left(\frac{\delta_\lambda^2+\delta_\mu^2}{2}\right)\\
&\ge f_t(\delta_\gamma^2)
=\cosh(\delta_\gamma t),
\end{align*}
where the last inequality uses
\begin{equation*}
\frac{\delta_\lambda^2+\delta_\mu^2}{2}
\ge\delta_\gamma^2.
\end{equation*}
This proves the theorem.
\end{proof}

Thus in two variables the second-order behavior at the collision is complete.  The next section shows that this phenomenon changes sharply in three variables: the same quadratic quantity remains necessary, but it no longer determines the sign of the full collision restriction.

\section{The Three-Variable Collision Criterion}
\label{sec:collision-criterion}

In three variables the quadratic collision quantity is no longer complete; the whole one-variable collision function is required.  We first introduce this function and show why second-order information alone is insufficient.  We then state the complete collision criterion and explain how its hypothesis can be verified exactly by Sturm's theorem.

\subsection{The collision function and second-order data}

For a symmetric three-variable polynomial, the locus where two variables coincide consists of the three sets \(x=y\), \(y=z\), and \(z=x\).  By symmetry it is enough to use one of them, and by homogeneity every positive point on that locus can be normalized to
\begin{equation*}
(t,1,1),
\qquad t>0.
\end{equation*}
For a fixed equal-degree chain
\begin{equation*}
\lambda\succ\gamma\succ\mu,
\end{equation*}
define the collision function by
\begin{equation*}
Q(t)
=\sum_{i=1}^3 t^{\lambda_i}
+\sum_{i=1}^3 t^{\mu_i}
-2\sum_{i=1}^3 t^{\gamma_i}.
\end{equation*}
The labeled-orbit convention gives
\begin{equation*}
P_{\lambda,\gamma,\mu}(t,1,1)=2Q(t).
\end{equation*}
Thus the collision condition is simply
\begin{equation*}
Q(t)\ge0
\qquad(t>0).
\end{equation*}

\begin{lemma}[Second-order collision data]\label{lem:collision-energy}
For every equal-degree three-variable chain,
\begin{equation*}
Q(1)=Q'(1)=0,
\qquad
Q''(1)=2\Energy(\lambda,\gamma,\mu).
\end{equation*}
\end{lemma}

\begin{proof}
At \(t=1\), each of the three exponent sums equals \(3\), so \(Q(1)=0\).  Equality of the total degrees gives
\begin{equation*}
Q'(1)=|\lambda|+|\mu|-2|\gamma|=0.
\end{equation*}
Finally,
\begin{align*}
Q''(1)
&=\sum_i\lambda_i(\lambda_i-1)
 +\sum_i\mu_i(\mu_i-1)
 -2\sum_i\gamma_i(\gamma_i-1)\\
&=\|\lambda\|_2^2+\|\mu\|_2^2-2\|\gamma\|_2^2\\
&=2\Energy(\lambda,\gamma,\mu).
\end{align*}
\end{proof}

Collision positivity therefore implies \(\Energy(\lambda,\gamma,\mu)\ge0\).  Unlike the two-variable case, this quadratic condition is not sufficient.

\begin{example}[The quadratic condition is not sufficient]\label{ex:quadratic-not-sufficient}
Let
\begin{equation*}
\lambda=(4,1,0),
\qquad
\gamma=(3,2,0),
\qquad
\mu=(3,1,1).
\end{equation*}
Then
\begin{equation*}
\lambda\succ\gamma\succ\mu
\end{equation*}
and
\begin{equation*}
\Energy(\lambda,\gamma,\mu)=1>0.
\end{equation*}
However,
\begin{align*}
Q(t)
&=t^4-t^3-2t^2+3t-1\\
&=(t-1)^2(t^2+t-1).
\end{align*}
Hence
\begin{equation*}
Q(t)<0
\end{equation*}
for
\begin{equation*}
0<t<\frac{\sqrt5-1}{2}.
\end{equation*}
Thus even strict positivity of the quadratic collision quantity does not imply collision positivity in three variables.
\end{example}

This example identifies the correct replacement for the quadratic criterion: in three variables one must control the entire function \(Q(t)\), not only its second derivative at \(t=1\).

\subsection{The complete criterion}

\begin{theorem}[Three-variable collision criterion]\label{thm:main}
Let
\begin{equation*}
\lambda\succ\gamma\succ\mu
\end{equation*}
be equal-degree nonnegative integer partitions with at most three parts.  Then
\begin{equation*}
P_{\lambda,\gamma,\mu}(x,y,z)\ge0
\qquad(x,y,z>0)
\end{equation*}
if and only if
\begin{equation*}
Q(t)\ge0
\qquad(t>0).
\end{equation*}
\end{theorem}

Necessity follows immediately by substituting \((t,1,1)\).  The remainder of the paper proves sufficiency by the sector decomposition of Section~\ref{sec:global-roadmap}; the final negative-offset one-sided completion is proved in Proposition~\ref{prop:partii-negative-r}.  The next subsection explains how the one-variable hypothesis can be checked exactly for any fixed integer chain.

\subsection{Checking the collision condition by Sturm's theorem}

For integer exponent partitions, \(Q(t)\) is an ordinary polynomial with integer coefficients.  Lemma~\ref{lem:collision-energy} shows that \(t=1\) is always a zero of multiplicity at least two.  Unless \(Q\) is identically zero, write
\begin{equation*}
Q(t)=(t-1)^2\widehat Q(t).
\end{equation*}
Then collision positivity is equivalent to
\begin{equation*}
\widehat Q(t)\ge0
\qquad(t>0).
\end{equation*}

Sturm's theorem gives an exact finite procedure for deciding this one-variable condition \cite{BPR2006}.  Construct a Sturm sequence for the square-free factors of \(\widehat Q\), use it to isolate all positive real zeros, and evaluate \(\widehat Q\) at one point in each interval between consecutive zeros.  The collision condition holds exactly when none of those interval signs is negative.  Thus, for any fixed integer chain \(\lambda\succ\gamma\succ\mu\), the collision hypothesis in Theorem~\ref{thm:main} can be checked by a finite univariate computation.

\subsection{The degree-nine example revisited}

Consider again
\begin{equation*}
\lambda=(9,0,0),
\qquad
\gamma=(7,1,1),
\qquad
\mu=(3,3,3).
\end{equation*}
Section~2 gave an independent three-variable proof.  The collision criterion reduces the same inequality to a one-variable calculation.  Since
\begin{equation*}
P_{\lambda,\gamma,\mu}(t,1,1)=2Q(t),
\end{equation*}
we obtain
\begin{align*}
Q(t)
&=t^9-2t^7+3t^3-4t+2\\
&=(t-1)^2H(t),
\end{align*}
where
\begin{equation*}
H(t)=t^7+2t^6+t^5-t^3-2t^2+2.
\end{equation*}
To apply the procedure of Section~4.3 explicitly, let \(S_0=H\), \(S_1=H'\), and successively take the negative Euclidean remainders.  After multiplying individual members by positive constants, the Sturm sequence is
\begin{align*}
S_0(t)&=t^7+2t^6+t^5-t^3-2t^2+2,\\
S_1(t)&=7t^6+12t^5+5t^4-3t^2-4t,\\
S_2(t)&=10t^5+10t^4+28t^3+64t^2-8t-98,\\
S_3(t)&=98t^4+294t^3+147t^2-343t-245,\\
S_4(t)&=-73t^3-129t^2+53t+148,\\
S_5(t)&=-4949t^2+60319t+49,\\
S_6(t)&=5917t-66,\\
S_7(t)&=-1.
\end{align*}
The signs at the two endpoints of \((0,\infty)\) are
\begin{equation*}
\begin{array}{c|cccccccc|c}
 &S_0&S_1&S_2&S_3&S_4&S_5&S_6&S_7&V\\ \hline
0^+&+&-&-&-&+&+&-&-&3\\
+\infty&+&+&+&+&-&-&+&-&3
\end{array}
\end{equation*}
Thus Sturm's theorem gives
\begin{equation*}
N_{(0,\infty)}(H)=V(0^+)-V(+\infty)=0.
\end{equation*}
Since \(H(0)=2>0\), it follows that
\begin{equation*}
H(t)>0
\qquad(t>0).
\end{equation*}
For this particular polynomial the same positivity can also be seen directly.  For \(t\ge1\),
\begin{equation*}
H(t)
=(t^7-t^3)+2(t^6-t^2)+t^5+2>0,
\end{equation*}
while for \(0<t\le1\),
\begin{equation*}
H(t)=(1-t^2)(2-t^3)+2t^6+t^7>0.
\end{equation*}
Equivalently, on \([0,1]\),
\begin{align*}
H(t)={}&2(1-t)^7+14t(1-t)^6+40t^2(1-t)^5+59t^3(1-t)^4\\
&+46t^4(1-t)^3+17t^5(1-t)^2+4t^6(1-t)+3t^7,
\end{align*}
whose coefficients are all positive.  Hence
\begin{equation*}
Q(t)\ge0
\qquad(t>0).
\end{equation*}
The full collision criterion therefore gives
\begin{equation*}
J_{(9,0,0)}(x,y,z)+J_{(3,3,3)}(x,y,z)
\ge
2J_{(7,1,1)}(x,y,z)
\end{equation*}
for all positive \(x,y,z\).

Thus a nontrivial degree-nine inequality in three variables is reduced to the positivity of a single one-variable polynomial.  This is the practical advantage of the collision criterion: once the boundary inequality is verified, no separate interior argument is required.

\section{Canonical Geometry of a Three-Part Dominance Chain}
\label{sec:canonical-geometry}

We now develop the long-proof reduction for the three-variable theorem. Throughout Sections~\ref{sec:canonical-geometry}--\ref{sec:global-assembly}, let
\begin{equation*}
P(x,y,z)
=
J_\lambda(x,y,z)+J_\mu(x,y,z)-2J_\gamma(x,y,z),
\end{equation*}
where
\begin{equation*}
\lambda\succ\gamma\succ\mu
\end{equation*}
are equal-degree nonnegative integer partitions with at most three parts, and assume
\begin{equation*}
P(t,1,1)\ge0
\qquad(t>0).
\end{equation*}

\subsection{Common translation}

\begin{lemma}[Common translation]
\label{lem:common-translation}
There is an integer \(m\ge0\) such that, after replacing each exponent vector \(\eta\) by
\begin{equation*}
\eta-m(1,1,1),
\end{equation*}
the upper endpoint has the form
\begin{equation*}
\lambda=(p,q,0),
\qquad
p\ge q\ge0.
\end{equation*}
The replacement preserves both collision positivity and global positivity.
\end{lemma}

\begin{proof}
Take \(m=\lambda_3\).  Equal degree and
\begin{equation*}
\lambda\succ\gamma\succ\mu
\end{equation*}
imply
\begin{equation*}
\lambda_3\le\gamma_3\le\mu_3.
\end{equation*}
Thus subtracting \(\lambda_3\) from every coordinate leaves all three exponent vectors nonnegative.  Common translation leaves every majorization partial-sum difference unchanged, so the dominance chain is preserved.  Finally,
\begin{equation*}
J_{\eta+m(1,1,1)}(x,y,z)=(xyz)^mJ_\eta(x,y,z),
\end{equation*}
so the defect is multiplied by the positive factor \((xyz)^m\), and its collision restriction by \(t^m\).  Neither sign condition changes.
\end{proof}

\subsection{Canonical transfer and lattice coordinates}
\label{sec:canonical-coordinates}

The remainder of the proof repeatedly compares the three exponent vectors along the chain
\begin{equation*}
\lambda\succ\gamma\succ\mu.
\end{equation*}
We therefore introduce coordinates that record the actual changes of the three coordinates.  There are two closely related descriptions.  The variables
\begin{equation*}
(\alpha,\delta;a,c)
\end{equation*}
are adapted to the two successive transfers
\(\lambda\to\gamma\) and \(\gamma\to\mu\), while
\begin{equation*}
(r,s,\kappa)
\end{equation*}
are adapted to the lattice moves used in Section~\ref{sec:monge}.

After Lemma~\ref{lem:common-translation}, write
\begin{equation*}
\lambda=(p,q,0),
\qquad
d:=p-q.
\end{equation*}
Thus \(d=\lambda_1-\lambda_2\) is the upper adjacent gap.

For the first step \(\lambda\to\gamma\), define
\begin{equation*}
\alpha:=\lambda_1-\gamma_1,
\qquad
\delta:=\gamma_3-\lambda_3=\gamma_3.
\end{equation*}
Since total degree is preserved, the second coordinate changes by
\begin{equation*}
\gamma_2-\lambda_2=\alpha-\delta.
\end{equation*}
Hence
\begin{equation*}
\gamma
=
(p-\alpha,\ q+\alpha-\delta,\ \delta).
\end{equation*}
Thus \(\alpha\) is the total amount removed from the first coordinate in passing from
\(\lambda\) to \(\gamma\), while \(\delta\) is the amount appearing in the third coordinate.

For the second step \(\gamma\to\mu\), define
\begin{equation*}
a:=\gamma_1-\mu_1,
\qquad
c:=\mu_3-\gamma_3.
\end{equation*}
Again using equality of total degrees,
\begin{equation*}
\mu_2-\gamma_2=a-c,
\end{equation*}
and therefore
\begin{equation*}
\mu
=
(p-\alpha-a,\ q+\alpha-\delta+a-c,\ \delta+c).
\end{equation*}
The cumulative changes from the upper endpoint to the lower endpoint are
\begin{equation*}
\beta:=\alpha+a=\lambda_1-\mu_1,
\qquad
\kappa:=\delta+c=\mu_3-\lambda_3=\mu_3.
\end{equation*}
Consequently
\begin{equation*}
\mu
=
(p-\beta,\ q+\beta-\kappa,\ \kappa).
\end{equation*}

For the moving middle partition it is more convenient to use the coordinate changes themselves:
\begin{equation*}
r:=\gamma_2-\lambda_2=\alpha-\delta,
\qquad
s:=\gamma_3=\delta.
\end{equation*}
Then
\begin{equation*}
\alpha=r+s
\end{equation*}
and
\begin{equation*}
\gamma_{r,s}
=
(p-r-s,\ q+r,\ s).
\end{equation*}
These variables have a direct lattice interpretation.  Increasing \(r\) by one transfers one unit from the first coordinate of \(\gamma\) to its second coordinate, while increasing \(s\) by one transfers one unit from the first coordinate to the third.  With \(\beta\) fixed, increasing \(\kappa\) by one transfers one unit in \(\mu\) from the second coordinate to the third.  These are exactly the three neighboring directions used in Section~\ref{sec:monge}.

Three coordinate gaps will occur repeatedly:
\begin{equation*}
U:=\gamma_1-\gamma_2=d-2r-s,
\end{equation*}
\begin{equation*}
H:=\gamma_1-\gamma_3=p-r-2s,
\end{equation*}
and
\begin{equation*}
V:=\mu_2-\mu_3=q+\beta-2\kappa.
\end{equation*}
We also write
\begin{equation*}
\nu:=\mu_1-\mu_2=d-2\beta+\kappa.
\end{equation*}
Thus \(U\) and \(H\) measure the two indicated separations inside the middle partition, while \(V\) and \(\nu\) are the two lower adjacent gaps.  In particular, the inequalities \(U\ge0\), \(H\ge0\), \(V\ge0\), and \(\nu\ge0\) are not auxiliary analytic assumptions; they are consequences of the weak ordering of the corresponding exponent triples.

For reference, the canonical parameters may therefore be read directly from the three partitions:
\begin{equation*}
\begin{array}{c|c}
\text{parameter} & \text{coordinate meaning}\\ \hline
d & \lambda_1-\lambda_2\\
\alpha & \lambda_1-\gamma_1\\
\delta=s & \gamma_3\\
r & \gamma_2-\lambda_2\\
a & \gamma_1-\mu_1\\
c & \mu_3-\gamma_3\\
\beta & \lambda_1-\mu_1\\
\kappa & \mu_3\\
U & \gamma_1-\gamma_2\\
H & \gamma_1-\gamma_3\\
V & \mu_2-\mu_3\\
\nu & \mu_1-\mu_2
\end{array}
\end{equation*}

Finally, for \(m\ge1\), write
\begin{equation*}
A_m(t):=1+t+\cdots+t^{m-1},
\qquad
A_0(t):=0.
\end{equation*}
The notation \(A_m\) will always denote this geometric-sum polynomial.

\subsection{The midpoint rectangle}

\begin{lemma}[Canonical phase decomposition]
\label{lem:phase-decomposition}
One has
\begin{equation*}
\frac{\lambda+\mu}{2}\succcurlyeq\gamma
\end{equation*}
if and only if
\begin{equation*}
a\le\alpha,
\qquad
c\le\delta.
\end{equation*}
Consequently \(P\ge0\) in this rectangle.
\end{lemma}

\begin{proof}
The first partial-sum condition is
\begin{equation*}
\frac{p+(p-\alpha-a)}2\ge p-\alpha,
\end{equation*}
equivalent to \(a\le\alpha\). Equality of total degrees reduces the second condition to
\begin{equation*}
\frac{\delta+c}{2}\le\delta,
\end{equation*}
equivalent to \(c\le\delta\).

We first prove the midpoint inequality directly.  For each labeled permutation
\begin{equation*}
\sigma\in S_3,
\end{equation*}
set
\begin{equation*}
M_{\lambda,\sigma}
=
\prod_{i=1}^{3}x_i^{\lambda_{\sigma(i)}},
\qquad
M_{\mu,\sigma}
=
\prod_{i=1}^{3}x_i^{\mu_{\sigma(i)}}.
\end{equation*}
By the arithmetic--geometric mean inequality,
\begin{align*}
\frac{M_{\lambda,\sigma}+M_{\mu,\sigma}}{2}
&\ge
\sqrt{M_{\lambda,\sigma}M_{\mu,\sigma}}\\
&=
\prod_{i=1}^{3}
 x_i^{(\lambda_{\sigma(i)}+\mu_{\sigma(i)})/2}.
\end{align*}
Summing over all \(\sigma\in S_3\) gives
\begin{equation*}
\frac{J_\lambda+J_\mu}{2}
\ge
J_{(\lambda+\mu)/2}.
\end{equation*}
Since the first part of the proof shows
\begin{equation*}
\frac{\lambda+\mu}{2}\succcurlyeq\gamma,
\end{equation*}
Muirhead's inequality yields
\begin{equation*}
J_{(\lambda+\mu)/2}
\ge
J_\gamma.
\end{equation*}
Therefore
\begin{equation*}
\frac{J_\lambda+J_\mu}{2}
\ge
J_{(\lambda+\mu)/2}
\ge
J_\gamma,
\end{equation*}
and hence \(P\ge0\) in the midpoint rectangle.
\end{proof}

The inequalities in Lemma~\ref{lem:phase-decomposition} divide the remaining proof into four regions. We use the following terminology throughout:
\begin{align*}
a\le\alpha,\quad c\le\delta
&\qquad\text{midpoint region},\\
a>\alpha,\quad c\le\delta
&\qquad\text{one-sided exterior region},\\
a\le\alpha,\quad c>\delta
&\qquad\text{reciprocal one-sided exterior region},\\
a>\alpha,\quad c>\delta
&\qquad\text{double-exterior region}.
\end{align*}
The midpoint region has just been settled. The next lemma shows that the two one-sided exterior regions are equivalent, so it suffices later to treat the orientation \(a>\alpha, c\le\delta\).

\subsection{Reciprocal complement}

\begin{lemma}[Reciprocal complement]
\label{lem:reciprocal}
Fix \(K\) at least as large as every exponent and define
\begin{equation*}
(\eta_1,\eta_2,\eta_3)^\vee
=
(K-\eta_3,\ K-\eta_2,\ K-\eta_1).
\end{equation*}
Then
\begin{equation*}
J_{\eta^\vee}(x,y,z)
=
(xyz)^KJ_\eta(x^{-1},y^{-1},z^{-1}).
\end{equation*}
Consequently
\begin{equation*}
P^\vee(x,y,z)
=
(xyz)^KP(x^{-1},y^{-1},z^{-1}),
\qquad
P^\vee(t,1,1)
=
t^KP(t^{-1},1,1).
\end{equation*}
In canonical transfer coordinates reciprocity exchanges
\begin{equation*}
(\alpha,\delta;a,c)
\longleftrightarrow
(\delta,\alpha;c,a)
\end{equation*}
up to common translation.
\end{lemma}

\begin{proof}
The orbit identity is termwise.  If \(\eta_1\ge\eta_2\ge\eta_3\), then
\begin{equation*}
K-\eta_3\ge K-\eta_2\ge K-\eta_1,
\end{equation*}
so reciprocal complement again produces a partition.  For equal-total-degree triples, majorization is equivalent to the first-part inequality together with the reversed last-part inequality.  Complementing and reversing the coordinates preserves these inequalities, hence
\begin{equation*}
\eta\succcurlyeq\zeta
\quad\Longleftrightarrow\quad
\eta^\vee\succcurlyeq\zeta^\vee.
\end{equation*}
Therefore the dominance chain is preserved.  The two displayed identities for \(P^\vee\) preserve collision and global positivity, and direct substitution yields
\begin{equation*}
(\alpha,\delta;a,c)\longleftrightarrow(\delta,\alpha;c,a)
\end{equation*}
up to common translation.
\end{proof}

\section{Lattice Structure and Reduction to Minimal Collision-Positive States}
\label{sec:monge}

We now vary the middle partition and the lower endpoint while keeping the upper endpoint
\begin{equation*}
\lambda=(p,q,0)
\end{equation*}
fixed.  For integers \(r,s,\beta,\kappa\), set
\begin{equation*}
\gamma_{r,s}
=
(p-r-s,\ q+r,\ s),
\end{equation*}
and
\begin{equation*}
\mu_\kappa
=
(p-\beta,\ q+\beta-\kappa,\ \kappa).
\end{equation*}
For every state for which these triples are admissible, define
\begin{equation*}
P_{r,s,\kappa}
=
J_\lambda+J_{\mu_\kappa}-2J_{\gamma_{r,s}},
\end{equation*}
and define the collision quotient \(Q_{r,s,\kappa}\) by
\begin{equation*}
P_{r,s,\kappa}(t,1,1)
=
2(t-1)^2Q_{r,s,\kappa}(t).
\end{equation*}
The purpose of this section is to identify the monotone lattice directions of
\(P_{r,s,\kappa}\) and \(Q_{r,s,\kappa}\), and then to reduce the proof to states that are minimal with respect to those directions.

\subsection{Legal states and finite-difference identities}

\begin{definition}[Legal state]
\label{def:legal-state}
A triple \((r,s,\kappa)\) is called \emph{legal}, for fixed \((p,q,\beta)\), if
\begin{equation*}
\lambda\succcurlyeq\gamma_{r,s}\succcurlyeq\mu_\kappa
\end{equation*}
and both \(\gamma_{r,s}\) and \(\mu_\kappa\) are nonnegative weakly decreasing triples.
Equivalently,
\begin{align*}
s&\ge0,
& r+s&\ge0,
& q+r-s&\ge0,
& d-2r-s&\ge0,\\
\beta-r-s&\ge0,
& \kappa-s&\ge0,
& d-2\beta+\kappa&\ge0,
& q+\beta-2\kappa&\ge0.
\end{align*}
A neighboring state is called legal if the shifted triple satisfies the same conditions.
\end{definition}

The first four inequalities state that \(\gamma_{r,s}\) is a partition dominated by \(\lambda\); the next two are exactly the two partial-sum inequalities for
\(\gamma_{r,s}\succcurlyeq\mu_\kappa\); and the last two state that \(\mu_\kappa\) is weakly decreasing.  Notice in particular that legality is a property of the lattice state itself, not an additional hypothesis introduced in the finite-difference formulas below.

For any lattice quantity \(F_{r,s,\kappa}\), define the forward difference operators by
\begin{align*}
\Delta_rF_{r,s,\kappa}
&=
F_{r+1,s,\kappa}-F_{r,s,\kappa},\\
\Delta_sF_{r,s,\kappa}
&=
F_{r,s+1,\kappa}-F_{r,s,\kappa},\\
\Delta_\kappa F_{r,s,\kappa}
&=
F_{r,s,\kappa+1}-F_{r,s,\kappa}.
\end{align*}
Thus
\begin{equation*}
\Delta_r\Delta_sF_{r,s,\kappa}
=
F_{r+1,s+1,\kappa}
-F_{r+1,s,\kappa}
-F_{r,s+1,\kappa}
+F_{r,s,\kappa}.
\end{equation*}
The positive monotonicity direction for the lower endpoint is \(-\Delta_\kappa\), because increasing \(\kappa\) moves \(\mu_\kappa\) downward in dominance order.

\begin{lemma}[Finite-difference monotonicity identities]
\label{lem:finite-difference-monotonicity}
Let \((r,s,\kappa)\) be legal.
Whenever the indicated neighboring state is also legal, one has
\begin{align*}
\Delta_rP_{r,s,\kappa}
&=
2\bigl(J_{\gamma_{r,s}}-J_{\gamma_{r+1,s}}\bigr)
\ge0,\\
\Delta_sP_{r,s,\kappa}
&=
2\bigl(J_{\gamma_{r,s}}-J_{\gamma_{r,s+1}}\bigr)
\ge0,
\end{align*}
and
\begin{equation*}
-\Delta_\kappa P_{r,s,\kappa}
=
J_{\mu_\kappa}-J_{\mu_{\kappa+1}}
\ge0.
\end{equation*}
At the collision quotient level,
\begin{align*}
\Delta_rQ_{r,s,\kappa}(t)
&=
2t^{q+r}A_{d-2r-s-1}(t),\\
\Delta_sQ_{r,s,\kappa}(t)
&=
2t^sA_{p-r-2s-1}(t),\\
-\Delta_\kappa Q_{r,s,\kappa}(t)
&=
t^\kappa A_{q+\beta-2\kappa-1}(t).
\end{align*}
Moreover, whenever the corresponding two-dimensional configuration is legal,
\begin{equation*}
\Delta_r\Delta_sQ_{r,s,\kappa}(t)
=
-2t^{p-r-s-2},
\end{equation*}
whereas the lower-endpoint direction is separable from the two middle-partition directions:
\begin{equation*}
\Delta_\kappa\Delta_rQ_{r,s,\kappa}
=
\Delta_\kappa\Delta_sQ_{r,s,\kappa}
=
0.
\end{equation*}
\end{lemma}

\begin{proof}
The three full-polynomial inequalities follow from Muirhead's inequality.  Indeed,
\(\gamma_{r,s}\) dominates \(\gamma_{r+1,s}\) by one unit transfer from the first coordinate to the second,
\(\gamma_{r,s}\) dominates \(\gamma_{r,s+1}\) by one unit transfer from the first coordinate to the third, and
\(\mu_\kappa\) dominates \(\mu_{\kappa+1}\) by one unit transfer from the second coordinate to the third.

Because the orbit sums are labeled, for any exponent triple \((u,v,w)\),
\begin{equation*}
J_{(u,v,w)}(t,1,1)
=
2(t^u+t^v+t^w).
\end{equation*}
For integers \(u>v\),
\begin{align*}
t^u+t^v-t^{u-1}-t^{v+1}
&=(t-1)(t^{u-1}-t^v)\\
&=(t-1)^2t^vA_{u-v-1}(t).
\end{align*}
Applying this identity to the three unit transfers above and dividing by
\(2(t-1)^2\) gives the three formulas for \(Q\).

For the mixed difference, use the first collision finite-difference formula:
\begin{align*}
\Delta_s\Delta_rQ_{r,s,\kappa}(t)
&=
2t^{q+r}
\left(
A_{d-2r-s-2}(t)-A_{d-2r-s-1}(t)
\right)\\
&=
-2t^{q+r}t^{d-2r-s-2}\\
&=
-2t^{p-r-s-2}.
\end{align*}
Finally, the formulas for \(\Delta_rQ\) and \(\Delta_sQ\) do not involve \(\kappa\), which gives
\begin{equation*}
\Delta_\kappa\Delta_rQ
=
\Delta_\kappa\Delta_sQ
=0.
\end{equation*}
\end{proof}

For each fixed \(t>0\), Lemma~\ref{lem:finite-difference-monotonicity} shows that
\(Q_{r,s,\kappa}(t)\) is nondecreasing in \(r\) and \(s\), and nonincreasing in \(\kappa\), along legal neighboring states.  In addition,
\begin{equation*}
\Delta_r\Delta_sQ_{r,s,\kappa}(t)<0.
\end{equation*}
We shall refer to this sign condition as the \emph{discrete Monge property} of the \((r,s)\)-array.  The vanishing mixed differences with \(\kappa\) mean that the lower-endpoint increment is independent of the two middle-partition increments.  These formulas, rather than any separate geometric slogan, are the lattice structure used in the reduction below.

\subsection{Reduction to minimal collision-positive states}

We write
\begin{equation*}
Q_{r,s,\kappa}\ge0
\end{equation*}
to mean
\begin{equation*}
Q_{r,s,\kappa}(t)\ge0
\qquad\text{for every }t>0.
\end{equation*}
Similarly,
\begin{equation*}
Q_{r,s,\kappa}\not\ge0
\end{equation*}
means that \(Q_{r,s,\kappa}(t)<0\) for at least one \(t>0\).

\begin{definition}[Predecessors and minimal collision-positive states]
\label{def:minimal-collision-positive}
For a legal state \((r,s,\kappa)\), its legal predecessors are the states
\begin{equation*}
(r-1,s,\kappa),
\qquad
(r,s-1,\kappa),
\qquad
(r,s,\kappa+1),
\end{equation*}
whenever they are legal.
A legal state \((r,s,\kappa)\) is called \emph{minimal collision-positive} if
\begin{equation*}
Q_{r,s,\kappa}\ge0
\end{equation*}
and every legal predecessor fails collision positivity.
\end{definition}

The terminology is consistent with the monotonicity directions in Lemma~\ref{lem:finite-difference-monotonicity}: moving from a predecessor to \((r,s,\kappa)\) adds a globally nonnegative full-polynomial increment.

\begin{theorem}[Minimal-State Reduction]
\label{thm:minimal-state-reduction}
Let \((r,s,\kappa)\) be a collision-positive legal state.  Then there exists a minimal collision-positive legal state
\begin{equation*}
(r_*,s_*,\kappa_*)
\end{equation*}
with
\begin{equation*}
r_*\le r,
\qquad
s_*\le s,
\qquad
\kappa_*\ge\kappa,
\end{equation*}
such that
\begin{equation*}
P_{r,s,\kappa}(x,y,z)
\ge
P_{r_*,s_*,\kappa_*}(x,y,z)
\qquad(x,y,z>0).
\end{equation*}
Consequently, if
\begin{equation*}
P_{r_*,s_*,\kappa_*}(x,y,z)\ge0
\qquad(x,y,z>0)
\end{equation*}
for every minimal collision-positive legal state, then the same conclusion holds for every collision-positive legal state.
\end{theorem}

\begin{proof}
Start from a collision-positive legal state \((r,s,\kappa)\).  If it is not minimal collision-positive, at least one of its legal predecessors is collision-positive.  Move to such a predecessor and repeat.

The legal lattice is finite: with the total degree and the upper endpoint fixed, only finitely many integer partitions can occur as the middle and lower exponent vectors.  In addition, every predecessor step strictly decreases the integer
\begin{equation*}
r+s-\kappa,
\end{equation*}
so no cycle is possible.  Hence the descent terminates at a minimal collision-positive legal state \((r_*,s_*,\kappa_*)\).

At every step Lemma~\ref{lem:finite-difference-monotonicity} gives a globally nonnegative difference.  For example,
\begin{align*}
P_{r,s,\kappa}-P_{r-1,s,\kappa}
&=
\Delta_rP_{r-1,s,\kappa}
\ge0,\\
P_{r,s,\kappa}-P_{r,s-1,\kappa}
&=
\Delta_sP_{r,s-1,\kappa}
\ge0,\\
P_{r,s,\kappa}-P_{r,s,\kappa+1}
&=
-\Delta_\kappa P_{r,s,\kappa}
\ge0.
\end{align*}
Telescoping along the finite predecessor chain therefore gives
\begin{equation*}
P_{r,s,\kappa}
\ge
P_{r_*,s_*,\kappa_*}
\end{equation*}
pointwise on \((0,\infty)^3\).  The final assertion follows immediately.
\end{proof}

The reduction theorem itself uses only the three nonnegative full-polynomial finite differences and finiteness of the legal lattice.  The discrete Monge identity is not needed for the finite descent; its role is to describe the structure of the set of minimal collision-positive states that remains after the reduction.

\subsection{Relative support for minimal states}
\label{sec:relative-support}

It remains to classify the minimal collision-positive states.  For a shifted geometric sum
\begin{equation*}
t^mA_n(t)=t^m+t^{m+1}+\cdots+t^{m+n-1},
\end{equation*}
we call the integer interval
\begin{equation*}
[m,m+n-1]
\end{equation*}
its exponent support.  The support polynomials appearing on the minimal-state set are of this form.  The parameter that governs their relative displacement is
\begin{equation*}
\chi
:=
q+r-\kappa
=
(\gamma_{r,s})_2-(\mu_\kappa)_3.
\end{equation*}
Thus \(\chi\) is the integer gap between the second exponent of the middle partition and the third exponent of the lower endpoint.  In the ratio of the two support polynomials used in the next section, this gap appears through the shift factor
\(t^{\chi-1}\); consequently the values \(\chi=1\) and \(\chi=0\) are the two boundary configurations between separated and overlapping supports.

We shall use the following terminology.  The four mutually exclusive cases
\begin{equation*}
\chi\ge2,
\qquad
\chi=1,
\qquad
\chi=0,
\qquad
\chi\le-1
\end{equation*}
are called, respectively, \emph{crossing support}, the \emph{seam}, \emph{contact}, and \emph{strict containment}.  These names refer only to the relative position of the collision supports.  They do not replace the phase decomposition of Section~\ref{sec:canonical-geometry}.  In particular, the double-exterior phase
\begin{equation*}
a>\alpha,
\qquad
c>\delta
\end{equation*}
is completed by the direct channel theorem for \(V\in\{0,1\}\) and the reciprocal maximal-depth argument below.

\section{Global Reduction and Sector Decomposition}
\label{sec:global-roadmap}
After the common translation of Lemma~\ref{lem:common-translation}, every admissible chain belongs to exactly one of
\begin{align*}
\text{(I)}\quad &a\le\alpha,\qquad c\le\delta,\\
\text{(II)}\quad &a>\alpha,\qquad c\le\delta,\\
\text{(III)}\quad &a\le\alpha,\qquad c>\delta,\\
\text{(IV)}\quad &a>\alpha,\qquad c>\delta.
\end{align*}
Part I is the midpoint region and is already proved.  Part III is the reciprocal image of Part II.  Thus only the preferred exterior orientation \(a>\alpha\) remains.

After Minimal-State Reduction, a minimal state is classified by \(\chi=q+r-\kappa\).  The remaining minimal states fall into three cases:
\begin{equation*}
\begin{array}{c|c}
\text{support geometry}&\text{proof mechanism}\\ \hline
\chi\ge2 & \text{ordered-minor crossing theorem}\\
\chi=1 & \text{ordered-ratio seam argument, or the direct channel lemma for \(V\in\{0,1\}\)}\\
\chi\le0 & \text{reciprocal maximal-depth completion.}
\end{array}
\end{equation*}
In double exterior, the only seam not covered by the ordered-minor theorem for \(V\ge2\) is \(V\in\{0,1\}\).  Its positive-\(r\) half is proved by the multiplicative-kernel argument of Section~\ref{sec:seam-v01}; its negative-\(r\) half is absorbed into the maximal-depth theorem of Section~\ref{sec:one-sided-completion}.  No separate contact, strict-containment, neutral-boundary, or legality-cone energy theorem is needed.

\section{Crossing Support and the Seam}
\label{sec:crossing-seam}

The next calculation is phase-independent at crossing support.  Once a minimal collision-positive state has \(\chi\ge2\), the argument depends only on the relative ordering of the two positive contribution supports and on the three spectator placements; it does not use the signs of \(a-\alpha\) or \(c-\delta\).  We therefore prove crossing support once and reuse it in every phase.  The seam \(\chi=1\) requires the additional boundary distinctions recorded below.

\label{sec:one-sided-crossing}

We first treat minimal collision-positive states for which the moving and terminal collision support polynomials overlap without containment. The argument has two parts: crossing support \(\chi\ge2\), and the limiting genuine seam \(\chi=1\) with \(V\ge2\).

\subsection{Crossing support: \texorpdfstring{\(\chi\ge2\)}{chi >= 2}}
\label{crossing-support-chige2}

At a minimal collision-positive state, the two collision finite differences that meet at the terminal edge are
\begin{equation*}
\mathcal A_r(t)
:=2t^{q+r-1}A_{U+1}(t),
\qquad
\mathcal A_\kappa(t)
:=t^\kappa A_{V-1}(t).
\end{equation*}
Their ratio is
\begin{equation*}
\rho_r(t)
:=\frac{\mathcal A_r(t)}{\mathcal A_\kappa(t)}
=2t^{\chi-1}\frac{A_{U+1}(t)}{A_{V-1}(t)}.
\end{equation*}
Thus the relative displacement of the two supports is exactly the exponent shift \(\chi-1\).  When \(V\ge2\), the displayed quotient is the convenient interior representation and the positive shift strictly orders the two nonzero support intervals.  When \(V=0\) or \(V=1\), that quotient is not used.  The corresponding boundary cases \(V\in\{0,1\}\) are handled by the boundary coefficient theorem in the companion technical supplement.

There are three possible choices of spectator variable, which we order as
\begin{equation*}
xy\mid z\;\succ\;xz\mid y\;\succ\;yz\mid x.
\end{equation*}
For an admissible moving index \(j\) and one of these three placements \(\sigma\), let \(\rho_j^\sigma\) denote the corresponding support ratio.  The adjacent moving-index comparison and the spectator-ratio comparison both have the same orientation.  Consequently, for every \(c_*>0\), the threshold set
\begin{equation*}
\mathcal F(c_*)
:=\{(j,\sigma):\rho_j^\sigma\ge c_*\}
\end{equation*}
is downward closed in the product of the moving-index order and the above spectator order.  We shall call such a set a \emph{downward-closed threshold set}.

The only operation needed to straighten such a threshold set is the following elementary uncrossing identity.  Let
\begin{equation*}
s_0,s_1,\ldots
\end{equation*}
be projectively ordered positive two-vectors and put
\begin{equation*}
K(i,j):=\det(s_i,s_j),
\qquad i<j.
\end{equation*}
Then \(K(i,j)\ge0\), and for
\begin{equation*}
a<c<b<t
\end{equation*}
the Pl\"ucker relation gives
\begin{equation*}
K(a,b)K(c,t)-K(a,t)K(c,b)
=K(a,c)K(b,t)\ge0.
\end{equation*}
Thus replacing one adjacent crossed pair by the nested pair changes the coefficient by a nonnegative quantity.  At a boundary where one of the ordered minors vanishes, the same identity holds with equality; hence no separate sign convention is required.

Starting from a downward-closed threshold set and repeatedly applying this adjacent uncrossing eventually produces a nested arrangement.  In that arrangement the coefficient expansion is triangular: every coefficient is a sum of products of binomial coefficients and ordered minors, and is therefore nonnegative.  Reversing the uncrossings adds only nonnegative corrections.

The coefficient representation, the three-family total-positivity theorem, the boundary calculation for \(V\in\{0,1\}\), and the terminal nested binomial expansion are given in full in the companion technical supplement.  In particular, the supplement supplies the bridge from the support-ratio threshold sets to the actual coefficients of the wall residual; the Pl\"ucker identity above is only the local uncrossing step.

\begin{theorem}[Crossing-Support Theorem]
\label{thm:crossing-support}
Let \(r,s,\kappa\) be a minimal collision-positive legal state. If
\begin{equation*}
\chi\ge2,
\end{equation*}
then
\begin{equation*}
P_{r,s,\kappa}(x,y,z)\ge0
\qquad(x,y,z>0).
\end{equation*}
\end{theorem}

\begin{proof}
For \(V\ge2\), the support ordering and successive uncrossing calculation above replace each crossed contribution by a more nested contribution plus a coefficientwise nonnegative correction and terminate in the positive nested coefficient configuration.  For \(V\le1\), the quotient representation is replaced by the boundary case \(V\in\{0,1\}\).  The full boundary calculation, including the boundary kernel and the terminal nested coefficient expansion, is given in the companion technical supplement.  These cases exhaust crossing support.
\end{proof}

\subsection{The seam: \texorpdfstring{\(\chi=1\)}{chi = 1}}
\label{one-sided-seam-chi1}

On the seam the shift factor disappears and
\begin{equation*}
\rho_r(t)
=2\frac{A_{U+1}(t)}{A_{V-1}(t)}.
\end{equation*}
The two support intervals are
\begin{equation*}
I_r=[0,U],
\qquad
I_\kappa=[0,V-2].
\end{equation*}
Using
\begin{equation*}
\nu=d-2\beta+\kappa,
\qquad
a=\beta-\alpha,
\qquad
\alpha=r+s,
\qquad
q+r=\kappa+1,
\end{equation*}
we obtain
\begin{equation*}
U-(V-2)=\nu+a+1>0.
\end{equation*}
Hence, when \(V\ge2\),
\begin{equation*}
[0,V-2]\subsetneq[0,U].
\end{equation*}
For a geometric sum,
\begin{equation*}
t\frac{d}{dt}\log A_m(t)
\end{equation*}
is the weighted mean of the exponents \(0,\ldots,m-1\).  Strict inclusion of the exponent intervals therefore gives
\begin{equation*}
t\frac{d}{dt}\log\rho_r(t)>0.
\end{equation*}
Thus \(\rho_r\) is strictly increasing.  The same adjacent moving-index and spectator-ratio comparisons used above remain valid, so the threshold sets are again downward-closed threshold sets.  Since the Pl\"ucker uncrossing identity is unchanged when the shift is zero, the crossing-support argument extends to the genuine seam.

If \(V\le1\), the terminal support polynomial has no nonzero geometric-sum factor and the quotient above is not used.  In the one-sided exterior region these are precisely the boundary cases \(V\in\{0,1\}\); the companion technical supplement proves their nonnegativity by the boundary coefficient theorem.  In the double-exterior region the low seam is a different geometric branch and is treated separately in Section~\ref{sec:seam-v01}.

\begin{theorem}[One-Sided Seam Theorem]
\label{thm:genuine-seam-one-sided}
Let \(r,s,\kappa\) be a minimal collision-positive legal state in the one-sided exterior region. If
\begin{equation*}
\chi=1,
\end{equation*}
then
\begin{equation*}
P_{r,s,\kappa}(x,y,z)\ge0
\qquad(x,y,z>0).
\end{equation*}
\end{theorem}

\begin{proof}
For \(V\ge2\), strict support inclusion and support-ratio monotonicity give the same support ordering as in crossing support, and the order-zero uncrossing applies.  For \(V\le1\), use the boundary cases \(V\in\{0,1\}\) of the completed coefficient theorem in the companion technical supplement.
\end{proof}

\begin{lemma}[Common-middle collision lifting]
\label{lem:common-middle-lifting}
Assume that a legal canonical chain satisfies
\begin{equation*}
a=c.
\end{equation*}
If
\begin{equation*}
P(t,1,1)\ge0
\qquad(t>0),
\end{equation*}
then
\begin{equation*}
P(x,y,z)\ge0
\qquad(x,y,z>0).
\end{equation*}
\end{lemma}

\begin{proof}
Put
\begin{equation*}
B:=q+r.
\end{equation*}
Since \(a=c\), the middle exponent is unchanged on the lower leg:
\begin{equation*}
\gamma=(p-r-s,B,s),
\qquad
\mu=(p-r-s-c,B,s+c).
\end{equation*}
For an exponent triple \((u,w,v)\), write
\begin{equation*}
\mathcal J_{u,w,v}^{XY\mid Z}
:=Z^w\bigl(X^uY^v+X^vY^u\bigr).
\end{equation*}
Summing over the three choices of spectator variable gives the labeled orbit sum, including the correct multiplicities when exponents coincide.

We split according to the sign of \(r\).

Suppose first that \(r\ge0\), and define
\begin{equation*}
\tau=(p-r,B,0).
\end{equation*}
Then
\begin{equation*}
P=(J_\lambda-J_\tau)+(J_\tau+J_\mu-2J_\gamma).
\end{equation*}
The second group has common spectator exponent \(B\).  For one channel put
\begin{equation*}
t=X/Y,
\qquad
\rho=Z/Y,
\end{equation*}
and divide by the common factor \(Y^N\), \(N=p+q\).  The channel has the form
\begin{equation}\label{eq:common-middle-positive-r-channel}
F(t,\rho)=e(t)+\rho^B h(t),
\end{equation}
where
\begin{equation}\label{eq:common-middle-positive-r-profile}
e(t)
=t^q(t^r-1)(t^{d-r}-1)\ge0
\qquad(t\ge1).
\end{equation}
The partition inequalities imply \(d-r\ge0\).

At the collision points \((t,1,1)\) and \((t,t,1)\), the channel with equal paired variables vanishes.  Hence collision positivity gives
\begin{equation*}
e(t)+h(t)\ge0,
\qquad
e(t)+t^Bh(t)\ge0.
\end{equation*}
For \(t\ge1\), the second inequality yields
\begin{equation*}
h(t)\ge-t^{-B}e(t),
\end{equation*}
so
\begin{equation}\label{eq:common-middle-positive-r-bound}
F(t,\rho)
\ge
\left[1-\left(\frac{\rho}{t}\right)^B\right]e(t).
\end{equation}
If \(B=0\), then \(q=r=0\), the first group vanishes, and collision positivity gives \(h(t)\ge0\) directly.  We may therefore assume \(B>0\).

Order the physical variables as
\begin{equation*}
A\ge B_0\ge C>0,
\qquad
x=A/B_0,
\qquad
y=B_0/C.
\end{equation*}
The three channels correspond to
\begin{equation*}
(t,\rho)
=(x,y^{-1}),
\qquad
(xy,y),
\qquad
(y,xy).
\end{equation*}
The first channel is nonnegative by \eqref{eq:common-middle-positive-r-bound}.  The middle and last channels have the same lower-variable factor \(C^N\), and
\begin{align*}
F(xy,y)&\ge(1-x^{-B})e(xy),\\
F(y,xy)&\ge-(x^B-1)e(y).
\end{align*}
Finally,
\begin{align*}
e(xy)
&=(xy)^q\bigl((xy)^r-1\bigr)
       \bigl((xy)^{d-r}-1\bigr)\\
&\ge x^{q+r}y^q(y^r-1)(y^{d-r}-1)\\
&=x^B e(y),
\end{align*}
so
\begin{equation*}
(1-x^{-B})e(xy)
\ge(x^B-1)e(y).
\end{equation*}
Thus the last possible deficit is paid by the middle channel.

Now suppose \(r<0\).  Write
\begin{equation*}
m:=-r>0,
\qquad
B=q-m.
\end{equation*}
Legality gives
\begin{equation*}
B=\gamma_2\ge\gamma_3=s\ge m,
\end{equation*}
and define
\begin{equation*}
\tau=(p,B,m).
\end{equation*}
Again
\begin{equation*}
P=(J_\lambda-J_\tau)+(J_\tau+J_\mu-2J_\gamma),
\end{equation*}
with the second group having common spectator exponent \(B\).  This time the first group fixes the exponent \(p\), and its pair profile is
\begin{equation}\label{eq:common-middle-negative-r-profile}
e(t)=(t^m-1)(t^B-1)\ge0
\qquad(t\ge1).
\end{equation}
Thus the normalized channel is
\begin{equation}\label{eq:common-middle-negative-r-channel}
F(t,\rho)=\rho^p e(t)+\rho^B h(t).
\end{equation}
The two collision endpoints give
\begin{equation*}
e(t)+h(t)\ge0,
\qquad
t^pe(t)+t^Bh(t)\ge0.
\end{equation*}
Since \(p>B\) and \(t\ge1\), the first bound is stronger, so
\begin{equation*}
h(t)\ge-e(t).
\end{equation*}
Consequently
\begin{equation}\label{eq:common-middle-negative-r-bound}
F(t,\rho)
\ge
\rho^B\bigl(\rho^{p-B}-1\bigr)e(t).
\end{equation}
Every channel with \(\rho\ge1\) is therefore nonnegative.  With the same ordered physical variables as above, only the first channel can lose.  Put
\begin{equation*}
a_0:=p-B>0.
\end{equation*}
After restoring the common factor \(C^N\), its deficit is at most
\begin{equation*}
C^N y^q(y^{a_0}-1)e(x),
\end{equation*}
whereas the middle channel contributes at least
\begin{equation*}
C^N y^B(y^{a_0}-1)e(xy).
\end{equation*}
Since \(q=B+m\), it remains only to note that
\begin{align*}
e(xy)
&=((xy)^m-1)((xy)^B-1)\\
&\ge y^m(x^m-1)(x^B-1)\\
&=y^m e(x).
\end{align*}
Hence the middle channel pays the entire first-channel deficit.  The third channel is nonnegative by \eqref{eq:common-middle-negative-r-bound}.  Therefore \(P\ge0\) in both signs of \(r\).
\end{proof}

\begin{lemma}[Seam lemma for \(V=0\)]
\label{lem:phase-independent-v0-seam}
Assume that a legal canonical state satisfies
\begin{equation*}
\chi=1,
\qquad
V=0,
\qquad
r\ge0.
\end{equation*}
If its collision restriction is nonnegative, then \(P\ge0\) on the positive orthant.
\end{lemma}

\begin{proof}
The identity
\begin{equation*}
a-c=-\chi+V
\end{equation*}
gives \(c=a+1\).  Put
\begin{equation*}
B:=q+r-1=\kappa.
\end{equation*}
Then
\begin{equation*}
\mu=(B+\nu,B,B),
\qquad
\gamma=(\gamma_1,B+1,s),
\qquad
\gamma_1=p-r-s.
\end{equation*}

First suppose \(r\ge1\).  Define
\begin{equation*}
\tau=(p-r+1,B,0),
\qquad
\gamma^+=(\gamma_1+1,B,s).
\end{equation*}
Then
\begin{equation}\label{eq:phase-independent-v0-decomposition}
P
=(J_\lambda-J_\tau)
 +(J_\tau+J_\mu-2J_{\gamma^+})
 +2(J_{\gamma^+}-J_\gamma).
\end{equation}
For one pair channel, with \(t=X/Y\) and \(\rho=Z/Y\), division by \(Y^N\) gives
\begin{equation}\label{eq:phase-independent-v0-channel}
F(t,\rho)
=e(t)+\rho^B h(t)+2\rho^s d_+(t),
\end{equation}
where
\begin{equation*}
e(t)=t^q(t^{r-1}-1)(t^{d-r+1}-1),
\end{equation*}
and
\begin{equation*}
d_+(t)=t^B(t-1)(t^M-1),
\qquad
M:=\gamma_1-B\ge1.
\end{equation*}
Both profiles are nonnegative for \(t\ge1\).  The collision endpoint \(\rho=t\) yields
\begin{equation*}
h(t)
\ge
-t^{-B}e(t)-2t^{s-B}d_+(t),
\end{equation*}
so
\begin{align}\label{eq:phase-independent-v0-bound}
F(t,\rho)\ge{}&
\left[1-\left(\frac{\rho}{t}\right)^B\right]e(t)\notag\\
&+2\rho^s
\left[1-\left(\frac{\rho}{t}\right)^{B-s}\right]d_+(t).
\end{align}
Here \(B-s=c>0\).  Thus the two channels with \(\rho\le t\) are nonnegative except that we retain the middle channel as the positive contribution that controls the unique overshooting channel.

Order \(A\ge B_0\ge C>0\) and put \(x=A/B_0\), \(y=B_0/C\).  From \eqref{eq:phase-independent-v0-bound},
\begin{align*}
F(xy,y)\ge{}&
(1-x^{-B})e(xy)
+2y^s(1-x^{s-B})d_+(xy),\\
F(y,xy)\ge{}&
-(x^B-1)e(y)
-2y^s(x^B-x^s)d_+(y).
\end{align*}
The two required multiplicative estimates are elementary:
\begin{align*}
e(xy)
&\ge x^{q+r-1}e(y)=x^B e(y),\\
d_+(xy)
&\ge x^{B+1}d_+(y)
\ge x^B d_+(y).
\end{align*}
They pay the two deficits term by term.  The remaining channel is nonnegative by \eqref{eq:phase-independent-v0-bound}.

It remains to treat \(r=0\).  Then \(B=q-1\) and
\begin{equation*}
\mu=(q+\nu-1,q-1,q-1).
\end{equation*}
Define
\begin{equation*}
\theta=(q+\nu-1,q,q-2).
\end{equation*}
Legality gives \(q\ge2\).  We have
\begin{equation*}
P=(J_\lambda+J_\theta-2J_\gamma)-(J_\theta-J_\mu).
\end{equation*}
The first group has common spectator exponent \(q\), while the second has spectator exponent
\begin{equation*}
L=q+\nu-1
\end{equation*}
and pair profile
\begin{equation*}
f(t)=t^{q-2}(t-1)^2.
\end{equation*}
Thus
\begin{equation*}
F(t,\rho)=\rho^q\bigl[c_0(t)-\rho^h f(t)\bigr],
\qquad
h:=\nu-1.
\end{equation*}
The two collision endpoints imply
\begin{equation*}
c_0(t)\ge f(t),
\qquad
c_0(t)\ge t^h f(t).
\end{equation*}
If \(\nu=1\), every channel is nonnegative.  If \(\nu\ge2\), then
\begin{equation*}
F(t,\rho)\ge\rho^q(t^h-\rho^h)f(t),
\end{equation*}
and the only overshooting loss is paid by
\begin{equation*}
f(xy)\ge x^q f(y),
\end{equation*}
which follows from
\begin{equation*}
\frac{f(xy)}{f(y)}
=x^{q-2}\left(\frac{xy-1}{y-1}\right)^2
\ge x^q.
\end{equation*}
Finally, if \(\nu=0\), the bound \(c_0\ge f\) is stronger for \(t\ge1\).  The only possible loss is in the channel with \(\rho<1\), and the adjacent channel pays it because
\begin{equation*}
\frac{f(xy)}{f(x)}
=y^{q-2}\left(\frac{xy-1}{x-1}\right)^2
\ge y^{q-2}.
\end{equation*}
Thus \(P\ge0\) also at \(r=0\).
\end{proof}

\begin{lemma}[Reciprocal one-sided seam]
\label{lem:reciprocal-one-sided-seam}
Let \((r,s,\kappa)\) be a minimal collision-positive legal state in the reciprocal one-sided exterior region
\begin{equation*}
a\le\alpha,
\qquad
c>\delta.
\end{equation*}
If
\begin{equation*}
\chi=1,
\end{equation*}
then
\begin{equation*}
P_{r,s,\kappa}(x,y,z)\ge0
\qquad(x,y,z>0).
\end{equation*}
\end{lemma}

\begin{proof}
For \(V\ge2\), the two nonzero positive contribution supports are
\begin{equation*}
I_r=[q+r-1,\ q+r-1+U],
\qquad
I_\kappa=[\kappa,\ \kappa+V-2].
\end{equation*}
On \(\chi=1\) they have a common left endpoint, while
\begin{equation*}
\max I_r-\max I_\kappa
=\nu+a+1>0.
\end{equation*}
Thus the same support-ratio monotonicity and downward-closed uncrossing as in the genuine seam apply.  Reciprocal complement reverses both the moving-index and spectator orders and rescales the two index directions by positive factors, so every \(3\times3\) minor keeps its sign.  Hence the seam with \(V\ge2\) is nonnegative.

It remains to treat \(V\le1\).  The identity
\begin{equation*}
a-c=-\chi+V
\end{equation*}
shows that \(V=1\) gives \(a=c\), so Lemma~\ref{lem:common-middle-lifting} applies.

If \(V=0\), then \(c=a+1\).  Since the present state is reciprocal one-sided,
\begin{equation*}
a\le\alpha,
\qquad
c>\delta.
\end{equation*}
The parameters are integral, hence \(c\ge\delta+1\), so
\begin{equation*}
a=c-1\ge\delta.
\end{equation*}
Therefore
\begin{equation*}
r=\alpha-\delta\ge0.
\end{equation*}
Lemma~\ref{lem:phase-independent-v0-seam} applies.  This proves the reciprocal one-sided seam without extending the scope of the coefficient theorem.
\end{proof}

\begin{corollary}[Exterior Crossing-and-Seam Corollary]
\label{thm:double-cross-seam}
Let \((r,s,\kappa)\) be a minimal collision-positive legal state in either the one-sided exterior or the double-exterior region.  If
\begin{equation*}
\chi\ge2,
\end{equation*}
or if
\begin{equation*}
\chi=1,\qquad V\ge2,
\end{equation*}
then
\begin{equation*}
P_{r,s,\kappa}(x,y,z)\ge0
\qquad(x,y,z>0).
\end{equation*}
\end{corollary}

\begin{proof}
Apply Theorem~\ref{thm:crossing-support} in crossing support.  On the seam with \(V\ge2\), the same ordered-minor argument used in Theorem~\ref{thm:genuine-seam-one-sided} is independent of the sign of \(c-\delta\), so it applies in either exterior region.
\end{proof}
\section{The Seam Cases \(V\in\{0,1\}\)}
\label{sec:seam-v01}

We now isolate the only seam with \(V\in\{0,1\}\) in the double-exterior region that is not covered by the ordered-minor theorem for \(V\ge2\).  Thus assume
\begin{equation*}
a>\alpha,\qquad c>\delta.
\end{equation*}
The crossing and genuine-seam cases were settled in Section~\ref{sec:crossing-seam}.  It remains, for nonnegative middle offset, to prove
\begin{equation*}
\chi=1,\qquad V\in\{0,1\},\qquad r\ge0.
\end{equation*}
When \(V=0\) or \(V=1\), the terminal geometric-sum factor disappears, so these two boundary cases require a separate direct argument.

Put
\begin{equation*}
k:=a-\alpha>0,
\qquad
u:=c-\delta=c-s>0,
\qquad
\nu:=d-2\beta+\kappa=\mu_1-\mu_2\ge0.
\end{equation*}
Since \(\chi=1\),
\begin{equation*}
q=2s+u-r+1,
\end{equation*}
while
\begin{equation*}
V=r+k-u+1.
\end{equation*}
Hence
\begin{equation*}
u=r+k+1-V,
\end{equation*}
so that
\begin{equation*}
V=1\iff c=a,
\qquad
V=0\iff c=a+1.
\end{equation*}
Substituting these relations gives
\begin{equation*}
q=2s+k+2-V,
\end{equation*}
\begin{equation*}
p=3r+4s+2k+\nu+1,
\end{equation*}
and the three exponent vectors become
\begin{equation*}
\lambda=
(3r+4s+2k+\nu+1,\ 2s+k+2-V,\ 0),
\end{equation*}
\begin{equation*}
\gamma=
(2r+3s+2k+\nu+1,\ r+2s+k+2-V,\ s),
\end{equation*}
\begin{equation*}
\mu=
(r+2s+k+\nu+1,\ r+2s+k+1,\ r+2s+k+1-V).
\end{equation*}
Thus only the two terminal faces \(V=1\) and \(V=0\) have to be considered.

\subsection{Pair channels and the multiplicative kernel}

For a labeled exponent triple \((\xi,\eta,\zeta)\), define the channel in which the exponent \(\eta\) is carried by the spectator variable \(Z\) by
\begin{equation*}
\mathcal J_{\xi,\eta,\zeta}^{XY\mid Z}
=
Z^\eta
\left(
X^\xi Y^\zeta+X^\zeta Y^\xi
\right).
\end{equation*}
Then, with the labeled orbit-sum convention,
\begin{equation*}
J_{(\xi,\eta,\zeta)}
=
\mathcal J_{\xi,\eta,\zeta}^{xy\mid z}
+
\mathcal J_{\xi,\eta,\zeta}^{xz\mid y}
+
\mathcal J_{\xi,\eta,\zeta}^{yz\mid x}.
\end{equation*}
This identity remains exact when two exponents coincide, because the six labeled permutations retain their multiplicities.

For a fixed channel put
\begin{equation*}
t=\frac XY,
\qquad
\rho=\frac ZY.
\end{equation*}
If every term in a group has the same spectator exponent \(B\), its channel has the form
\begin{equation*}
Y^N\rho^B h(t),
\end{equation*}
where \(N\) is the common total degree.

We shall repeatedly use the following elementary inequality.

\begin{lemma}[Multiplicative two-factor kernel]
\label{lem:multiplicative-kernel}
Let
\begin{equation*}
E(t)=t^b(t^\ell-1)(t^m-1),
\qquad
b\ge0,\qquad m\ge\ell\ge0.
\end{equation*}
Then, for \(x,y\ge1\),
\begin{equation*}
E(xy)\ge x^{b+\ell}E(y).
\end{equation*}
\end{lemma}

\begin{proof}
For \(\ell=0\) the statement is trivial.  Assume \(\ell>0\).  Since
\begin{equation*}
(xy)^\ell-1
=
x^\ell(y^\ell-1)+(x^\ell-1)
\ge
x^\ell(y^\ell-1),
\end{equation*}
and
\begin{equation*}
(xy)^m-1\ge y^m-1,
\end{equation*}
multiplication gives
\begin{equation*}
E(xy)
\ge
x^{b+\ell}y^b
(y^\ell-1)(y^m-1)
=
x^{b+\ell}E(y).
\end{equation*}
\end{proof}

We order the physical variables as
\begin{equation*}
A\ge B\ge C>0
\end{equation*}
and write
\begin{equation*}
x=\frac AB\ge1,
\qquad
y=\frac BC\ge1.
\end{equation*}
The three channels then have
\begin{equation*}
(A,B)\mid C:
\qquad
(t,\rho)=\left(x,\frac1y\right),
\end{equation*}
\begin{equation*}
(A,C)\mid B:
\qquad
(t,\rho)=(xy,y),
\end{equation*}
and
\begin{equation*}
(B,C)\mid A:
\qquad
(t,\rho)=(y,xy).
\end{equation*}
Only the last channel can satisfy \(\rho>t\).  The proof below therefore consists of using the positive contribution of the middle channel and using Lemma~\ref{lem:multiplicative-kernel} to pay the unique overshoot.

\subsection{The face \texorpdfstring{\(V=1\)}{V = 1}}

Assume first
\begin{equation*}
V=1.
\end{equation*}
Then
\begin{equation*}
q=2s+k+1.
\end{equation*}
Set
\begin{equation*}
B_0:=q+r=r+2s+k+1.
\end{equation*}
The middle exponents of \(\gamma\) and \(\mu\) are both \(B_0\).  Introduce
\begin{equation*}
\tau=(p-r,B_0,0).
\end{equation*}
Then
\begin{equation*}
P
=
\underbrace{(J_\lambda-J_\tau)}_{E}
+
\underbrace{(J_\tau+J_\mu-2J_\gamma)}_{H}.
\end{equation*}
The group \(H\) has common labeled spectator exponent \(B_0\).

For one ordered physical pair \(X\ge Y\), the first group has profile
\begin{equation*}
e(t)
=
t^q(t^r-1)(t^{d-r}-1).
\end{equation*}
Since \(r\ge0\), \(e(t)\ge0\) for \(t\ge1\).  On the present face,
\begin{equation*}
d-r=2r+2s+k+\nu\ge r,
\end{equation*}
so Lemma~\ref{lem:multiplicative-kernel} gives
\begin{equation*}
e(xy)\ge x^{q+r}e(y)=x^{B_0}e(y).
\end{equation*}

Let \(h(t)\) be the pair profile of \(H\).  The normalized channel is
\begin{equation*}
F(t,\rho)=e(t)+\rho^{B_0}h(t).
\end{equation*}
Collision positivity gives the two endpoint inequalities
\begin{equation*}
e(t)+h(t)\ge0,
\qquad
e(t)+t^{B_0}h(t)\ge0.
\end{equation*}
Since \(e(t)\ge0\), the second gives
\begin{equation*}
h(t)\ge-t^{-B_0}e(t),
\end{equation*}
and hence
\begin{equation*}
F(t,\rho)
\ge
\left[
1-\left(\frac{\rho}{t}\right)^{B_0}
\right]e(t).
\end{equation*}
Thus every channel with \(\rho\le t\) is nonnegative.

For the middle channel,
\begin{equation*}
F(xy,y)
\ge
(1-x^{-B_0})e(xy),
\end{equation*}
whereas the overshooting channel satisfies
\begin{equation*}
F(y,xy)
\ge
-(x^{B_0}-1)e(y).
\end{equation*}
The two channels have the same lower-variable factor \(C^N\), and
\begin{equation*}
(1-x^{-B_0})e(xy)
\ge
(x^{B_0}-1)e(y)
\end{equation*}
by the multiplicative-kernel inequality.  Therefore their sum is nonnegative; the remaining channel is already nonnegative.  This proves the face \(V=1\).

\subsection{The face \texorpdfstring{\(V=0\)}{V = 0} with \texorpdfstring{\(r\ge1\)}{r >= 1}}

Now assume
\begin{equation*}
V=0,
\qquad
r\ge1.
\end{equation*}
Then
\begin{equation*}
q=2s+k+2.
\end{equation*}
Put
\begin{equation*}
B_0:=q+r-1=r+2s+k+1.
\end{equation*}
Here
\begin{equation*}
\mu_2=\mu_3=B_0,
\qquad
\gamma_2=B_0+1.
\end{equation*}
Define
\begin{equation*}
\tau=(p-r+1,B_0,0),
\qquad
\gamma^+=(\gamma_1+1,B_0,s).
\end{equation*}
Then
\begin{equation*}
P
=
\underbrace{(J_\lambda-J_\tau)}_{E}
+
\underbrace{(J_\tau+J_\mu-2J_{\gamma^+})}_{H}
+
2\underbrace{(J_{\gamma^+}-J_\gamma)}_{D}.
\end{equation*}
The group \(H\) has common spectator exponent \(B_0\), while the group \(D\) has spectator exponent \(s\).

The first positive profile is
\begin{equation*}
e(t)
=
t^q(t^{r-1}-1)(t^{d-r+1}-1).
\end{equation*}
Since
\begin{equation*}
d-r+1=2r+2s+k+\nu\ge r-1,
\end{equation*}
Lemma~\ref{lem:multiplicative-kernel} gives
\begin{equation*}
e(xy)\ge x^{q+r-1}e(y)=x^{B_0}e(y).
\end{equation*}

The edge \(D\) has profile
\begin{equation*}
d_+(t)=t^{B_0}(t-1)(t^m-1),
\end{equation*}
where
\begin{equation*}
m=\gamma_1-B_0=r+s+k+\nu\ge1.
\end{equation*}
Hence
\begin{equation*}
d_+(xy)\ge x^{B_0+1}d_+(y)\ge x^{B_0}d_+(y).
\end{equation*}

If \(h(t)\) denotes the pair profile of \(H\), the full normalized channel is
\begin{equation*}
F(t,\rho)
=
e(t)+\rho^{B_0}h(t)+2\rho^s d_+(t).
\end{equation*}
Collision positivity at the endpoint \(\rho=t\) gives
\begin{equation*}
h(t)
\ge
-t^{-B_0}e(t)-2t^{s-B_0}d_+(t).
\end{equation*}
Therefore
\begin{equation*}
\begin{aligned}
F(t,\rho)\ge{}&
\left[
1-\left(\frac{\rho}{t}\right)^{B_0}
\right]e(t)\\
&+
2\rho^s
\left[
1-\left(\frac{\rho}{t}\right)^{B_0-s}
\right]d_+(t).
\end{aligned}
\end{equation*}
Since \(B_0>s\), this is nonnegative whenever \(\rho\le t\).

For the middle channel the displayed lower bound is
\begin{equation*}
\begin{aligned}
F(xy,y)\ge{}&
(1-x^{-B_0})e(xy)\\
&+
2y^s(1-x^{s-B_0})d_+(xy),
\end{aligned}
\end{equation*}
whereas the overshooting channel satisfies
\begin{equation*}
\begin{aligned}
F(y,xy)\ge{}&
-(x^{B_0}-1)e(y)\\
&-
2y^s(x^{B_0}-x^s)d_+(y).
\end{aligned}
\end{equation*}
The first deficit is paid by
\begin{equation*}
e(xy)\ge x^{B_0}e(y),
\end{equation*}
and the second by
\begin{equation*}
d_+(xy)\ge x^{B_0}d_+(y).
\end{equation*}
Thus the two channels sum to a nonnegative quantity, and the third channel is nonnegative separately.  This proves \(V=0,r\ge1\).

\subsection{The corner \texorpdfstring{\(V=0,r=0\)}{V = 0, r = 0}}

The preceding decomposition uses \(r-1\) and therefore does not apply at \(r=0\).  In this final corner,
\begin{equation*}
q=2s+k+2,
\end{equation*}
and
\begin{equation*}
\mu=(q+\nu-1,q-1,q-1).
\end{equation*}
Introduce
\begin{equation*}
\theta=(q+\nu-1,q,q-2).
\end{equation*}
Then
\begin{equation*}
P
=
\underbrace{(J_\lambda+J_\theta-2J_\gamma)}_{C_0}
-
\underbrace{(J_\theta-J_\mu)}_{F_0}.
\end{equation*}
The group \(C_0\) has common labeled spectator exponent \(q\).  The second group is the terminal balancing edge
\begin{equation*}
(q,q-2)\longrightarrow(q-1,q-1)
\end{equation*}
with spectator exponent
\begin{equation*}
L=q+\nu-1.
\end{equation*}
Its normalized pair profile is
\begin{equation*}
f(t)=t^{q-2}(t-1)^2.
\end{equation*}

Let \(c_0(t)\) be the pair profile of the common-\(q\) group and put
\begin{equation*}
h=L-q=\nu-1.
\end{equation*}
Then
\begin{equation*}
F(t,\rho)
=
\rho^q
\left[
c_0(t)-\rho^h f(t)
\right].
\end{equation*}
The two collision endpoints give
\begin{equation*}
c_0(t)\ge f(t),
\qquad
c_0(t)\ge t^h f(t).
\end{equation*}

If \(\nu=1\), then \(h=0\), so every channel is nonnegative immediately.

Assume \(\nu\ge2\).  Then \(h>0\), and
\begin{equation*}
F(t,\rho)
\ge
\rho^q(t^h-\rho^h)f(t).
\end{equation*}
Only the overshooting channel can be negative.  The compensation reduces to
\begin{equation*}
f(xy)\ge x^qf(y),
\end{equation*}
which follows from
\begin{equation*}
\frac{f(xy)}{f(y)}
=
x^{q-2}
\left(
\frac{xy-1}{y-1}
\right)^2
\ge x^q.
\end{equation*}

Finally suppose \(\nu=0\).  Then \(h=-1\), and the first collision bound is the stronger one.  Every channel with \(\rho\ge1\) is nonnegative.  The only possible loss occurs in the channel \((A,B)\mid C\), and it is paid by the adjacent channel \((A,C)\mid B\).  Since the common degree is
\begin{equation*}
N=3q-3,
\end{equation*}
the required comparison is exactly
\begin{equation*}
f(xy)\ge y^{q-2}f(x).
\end{equation*}
But
\begin{equation*}
\frac{f(xy)}{f(x)}
=
y^{q-2}
\left(
\frac{xy-1}{x-1}
\right)^2
\ge y^{q-2}.
\end{equation*}
Thus the last corner is also nonnegative.

\begin{theorem}[Low-positive contribution seam theorem for positive \(r\)]
\label{thm:seam-v01-positive-r}
Assume
\begin{equation*}
a>\alpha,\qquad
c>\delta,\qquad
\chi=1,\qquad
V\in\{0,1\},\qquad
r\ge0.
\end{equation*}
If
\begin{equation*}
P(t,1,1)\ge0
\qquad(t>0),
\end{equation*}
then
\begin{equation*}
P(x,y,z)\ge0
\qquad(x,y,z>0).
\end{equation*}
\end{theorem}

The three cases above exhaust the seam cases \(V\in\{0,1\}\) with nonnegative \(r\).  Together with Section~\ref{sec:crossing-seam}, this supplies all positive-\(r\) seam bases needed by the reciprocal maximal-depth argument below.

\section{Low-Support Exterior Completion by Reciprocal Maximal-Depth Descent}
\label{sec:one-sided-completion}

The preceding sections have proved the midpoint region, crossing support, the one-sided seam, the genuine double-exterior seam, and the seam cases \(V\in\{0,1\}\) with nonnegative \(r\).  We now show that no separate contact, strict-containment, neutral-boundary, or legality-cone energy calculation is needed.  All remaining exterior states in the boundary configurations are forced by reciprocal complement and the finite predecessor lattice.

\subsection{A common-middle lifting lemma}

The maximal-depth argument uses only the common-middle configuration \(a=c\).  Lemma~\ref{lem:common-middle-lifting} proves this case directly, so no separate transfer-configuration terminology or chamber involution is needed.

\subsection{The only new terminal seam base}

\begin{lemma}[Negative-offset terminal seam]
\label{lem:negative-terminal-seam}
Assume
\begin{equation*}
a>\alpha,
\qquad c>\delta,
\qquad \chi=1,
\qquad V=0,
\qquad r<0,
\qquad \nu\in\{0,1\}.
\end{equation*}
If the collision restriction is nonnegative, then \(P\ge0\) on the positive orthant.
\end{lemma}

\begin{proof}
Write
\begin{equation*}
r=-m<0,
\qquad
k=a-\alpha>0,
\qquad
B:=q-m.
\end{equation*}
The identities \(\chi=1\) and \(V=0\) give
\begin{equation*}
q=2s+k+2,
\qquad
p=-3m+4s+2k+\nu+1.
\end{equation*}
Legality gives \(s\ge m\) and \(k\ge m\), hence \(B\ge2\) and \(p\ge B\).  The three exponent vectors are
\begin{equation*}
\lambda=(p,q,0),
\end{equation*}
\begin{equation*}
\gamma=(-2m+3s+2k+\nu+1,\ B,\ s),
\end{equation*}
and
\begin{equation*}
\mu=(B+\nu-1,\ B-1,\ B-1).
\end{equation*}
Introduce
\begin{equation*}
\tau=(p,B,m),
\qquad
\theta=(B+\nu-1,B,B-2).
\end{equation*}
Since orbit sums are symmetric in the exponent labels, no ordering convention for \(\theta\) is needed.  We have the exact decomposition
\begin{equation}\label{eq:terminal-seam-decomposition}
P=(J_\lambda-J_\tau)
 +(J_\tau+J_\theta-2J_\gamma)
 -(J_\theta-J_\mu).
\end{equation}

Fix one pair channel, put \(t=X/Y\) and \(\rho=Z/Y\), and divide by the common positive factor \(Y^N\), where \(N\) is the total degree.  The three groups in \eqref{eq:terminal-seam-decomposition} have respectively the profiles
\begin{equation*}
e(t)=(t^B-1)(t^m-1),
\qquad
h(t),
\qquad
f(t)=t^{B-2}(t-1)^2,
\end{equation*}
with spectator exponents \(p,B,B+\nu-1\).  Thus the normalized channel is
\begin{equation}\label{eq:terminal-seam-channel}
\Phi(t,\rho)
 =\rho^p e(t)+\rho^B h(t)-\rho^{B+\nu-1}f(t).
\end{equation}
Because \(e(1)=h(1)=f(1)=0\), collision positivity on \((t,1,1)\) gives the single channel inequality
\begin{equation}\label{eq:terminal-seam-collision}
e(t)+h(t)-f(t)\ge0.
\end{equation}

First suppose \(\nu=1\).  Put \(a_0=p-B\ge0\) and \(g=h-f\).  From \eqref{eq:terminal-seam-collision}, \(g\ge-e\), hence
\begin{equation*}
\Phi(t,\rho)
 =\rho^B(\rho^{a_0}e+g)
 \ge \rho^B(\rho^{a_0}-1)e.
\end{equation*}
Thus every channel with \(\rho\ge1\) is nonnegative.

Order the physical variables as \(A\ge B_0\ge C>0\) and put
\begin{equation*}
x=A/B_0,
\qquad y=B_0/C.
\end{equation*}
The three normalized channels have
\begin{equation*}
(t,\rho)=(x,y^{-1}),\qquad(xy,y),\qquad(y,xy).
\end{equation*}
Only the first can be negative.  After restoring the common \(C^N\) factor, its deficit is at most
\begin{equation*}
y^q(y^{a_0}-1)e(x),
\end{equation*}
whereas the middle channel contributes at least
\begin{equation*}
y^B(y^{a_0}-1)e(xy).
\end{equation*}
Since \(q=B+m\) and
\begin{equation}\label{eq:e-multiplicative}
e(xy)\ge y^m e(x)
\end{equation}
for \(x,y\ge1\), the middle channel pays the entire deficit.  The third channel is nonnegative.

Now suppose \(\nu=0\).  From \eqref{eq:terminal-seam-collision}, \(h\ge f-e\), so
\begin{equation*}
\Phi(t,\rho)
\ge
\rho^{B-1}
\bigl[(\rho^{a_0+1}-\rho)e(t)+(\rho-1)f(t)\bigr].
\end{equation*}
Again all channels with \(\rho\ge1\) are nonnegative.  The first channel loses at most
\begin{equation*}
y^q(y^{a_0}-1)e(x)
+y^{p+m}(y-1)f(x),
\end{equation*}
while the middle channel contributes at least
\begin{equation*}
y^B(y^{a_0}-1)e(xy)
+y^{B-1}(y-1)f(xy).
\end{equation*}
The first deficit is paid by \eqref{eq:e-multiplicative}.  For the second, the parameter identities give
\begin{equation*}
p+m-B+1=B-2,
\end{equation*}
and therefore
\begin{equation*}
\frac{f(xy)}{f(x)}
 =y^{B-2}
 \left(\frac{xy-1}{x-1}\right)^2
 \ge y^{B-2}.
\end{equation*}
Hence
\begin{equation*}
y^{B-1}f(xy)\ge y^{p+m}f(x),
\end{equation*}
which pays the second deficit.  Summing the three pair channels proves \(P\ge0\).
\end{proof}

\subsection{Reciprocal depth and predecessor geometry}

For every exterior state define
\begin{equation*}
D:=a+c=\beta+\kappa-r-2s.
\end{equation*}
With the upper endpoint fixed, only finitely many legal middle and lower partitions exist, so \(D\) is bounded above.  Every ordinary predecessor strictly increases it:
\begin{equation}\label{eq:D-predecessors}
\begin{aligned}
D(r-1,s;\beta,\kappa)&=D+1,\\
D(r,s-1;\beta,\kappa)&=D+2,\\
D(r,s;\beta,\kappa+1)&=D+1.
\end{aligned}
\end{equation}

Under reciprocal complement,
\begin{equation}\label{eq:residual-reciprocal-dictionary}
r^\vee=-r,
\qquad
V^\vee=\nu,
\qquad
\nu^\vee=V,
\qquad
\chi^\vee=-\chi+V+\nu.
\end{equation}
Write a state as \(X=(r,s;\beta,\kappa)\), let \(Y=X^\vee\), and denote the three predecessors of \(Y\) by \(Y_{r^-},Y_{s^-},Y_{\kappa^+}\).  Direct substitution gives
\begin{equation}\label{eq:reciprocal-predecessors}
\begin{aligned}
(Y_{r^-})^\vee&=(r+1,s-1;\beta,\kappa),\\
(Y_{s^-})^\vee&=(r,s-1;\beta,\kappa),\\
(Y_{\kappa^+})^\vee&=(r,s;\beta+1,\kappa).
\end{aligned}
\end{equation}
The middle state is exactly the \(s\)-predecessor of \(X\), whereas the other two mapped-back states satisfy
\begin{equation}\label{eq:D-reciprocal-predecessors}
D((Y_{r^-})^\vee)=D(X)+1,
\qquad
D((Y_{\kappa^+})^\vee)=D(X)+1.
\end{equation}
We shall also use the elementary identity
\begin{equation}\label{eq:a-c-support}
a-c=-\chi+V.
\end{equation}

\subsection{The global residual theorem}

Call a state \emph{residual} if
\begin{equation}\label{eq:residual-class-simplified}
a>\alpha,
\qquad c\ge0,
\end{equation}
and either
\begin{equation*}
\chi\le0,
\end{equation*}
or
\begin{equation*}
c>\delta,
\qquad \chi=1,
\qquad V\in\{0,1\},
\qquad r<0.
\end{equation*}
Thus every contact or contained state in the preferred exterior orientation is residual, irrespective of the sign of \(r\); the second alternative is precisely the seam cases \(V\in\{0,1\}\) with negative \(r\).

\begin{proposition}[Completion of the remaining exterior cases]
\label{prop:partii-negative-r}
Every collision-positive residual state is globally nonnegative.
\end{proposition}

\begin{proof}
Suppose that a residual collision-positive counterexample exists.  By Theorem~\ref{thm:minimal-state-reduction}, descend to a minimal collision-positive state.  The predecessor formulas show that the condition \(a>\alpha\) is preserved.  A state with \(\chi\le0\) stays in \(\chi\le0\); a negative-\(r\) state in a seam case with \(V\in\{0,1\}\) either stays on that seam or enters \(\chi\le0\).  Hence a counterexample yields a minimal residual counterexample.

Among all minimal residual counterexamples with the same upper endpoint, choose \(X\) with maximal \(D(X)\).  Put \(Y=X^\vee\).

First suppose that \(X\) lies in the contained/contact branch \(\chi_X\le0\).  Write
\begin{equation*}
e:=-\chi_X\ge0,
\qquad
S:=e+V_X+\nu_X=\chi_Y.
\end{equation*}
If \(S=0\), then \(e=V_X=\nu_X=0\).  By \eqref{eq:a-c-support}, \(a_X=c_X\), while \(V_X=\nu_X=0\) makes the lower endpoint completely balanced.  Thus the lower leg is a single root-\(13\) transfer and Lemma~\ref{lem:common-middle-lifting} applies directly to \(X\).

Assume \(S\ge1\).  If \(Y\) is minimal and \(S\ge2\), Theorem~\ref{thm:crossing-support} applies.  It remains to discuss the minimal case \(S=1\).  If \(\nu_X=1\), then \(e=V_X=0\), so \eqref{eq:a-c-support} again gives \(a_X=c_X\), and Lemma~\ref{lem:common-middle-lifting} applies.  Hence \(\nu_X=0\).  The reciprocal state satisfies
\begin{equation*}
\chi_Y=1,
\qquad
V_Y=0,
\qquad
\nu_Y=V_X\in\{0,1\}.
\end{equation*}
If \(Y\) is reciprocal one-sided, Lemma~\ref{lem:reciprocal-one-sided-seam} applies.  If \(Y\) is double exterior and \(r_Y\ge0\), apply Theorem~\ref{thm:seam-v01-positive-r}.  If \(Y\) is double exterior and \(r_Y<0\), apply Lemma~\ref{lem:negative-terminal-seam}.  Thus a minimal reciprocal state never yields a counterexample.

Now suppose that \(X\) lies on the seam cases \(V\in\{0,1\}\) with negative \(r\).  If \(V_X=0\) and \(\nu_X\in\{0,1\}\), Lemma~\ref{lem:negative-terminal-seam} applies directly.  If \(V_X=1\) and \(\nu_X=0\), then \eqref{eq:a-c-support} gives \(a_X=c_X\), so Lemma~\ref{lem:common-middle-lifting} applies.  In the remaining cases, if \(Y\) is minimal then \eqref{eq:residual-reciprocal-dictionary} places it either on the genuine seam, on the seam cases \(V\in\{0,1\}\) with nonnegative \(r\), or in crossing support.  These are Corollary~\ref{thm:double-cross-seam}, Theorem~\ref{thm:seam-v01-positive-r}, and Theorem~\ref{thm:crossing-support}, respectively.

It remains only to treat the case in which \(Y\) is not minimal.  Let \(Y'\) be a collision-positive legal predecessor.  The predecessor \(Y_{s^-}\) is impossible, because by \eqref{eq:reciprocal-predecessors} its reciprocal is exactly the legal \(s\)-predecessor of the minimal state \(X\).  Hence \(Y'\) is either \(Y_{r^-}\) or \(Y_{\kappa^+}\).  Put
\begin{equation*}
Z:=(Y')^\vee.
\end{equation*}
Then \eqref{eq:D-reciprocal-predecessors} gives
\begin{equation*}
D(Z)=D(X)+1.
\end{equation*}
Apply Minimal-State Reduction to \(Z\), and let \(W\) be the resulting minimal collision-positive descendant.  By \eqref{eq:D-predecessors},
\begin{equation*}
D(W)\ge D(Z)>D(X).
\end{equation*}
If \(W\) is residual, maximality of \(D(X)\) implies \(P_W\ge0\).  If \(W\) is not residual, then, being minimal and still satisfying \(a>\alpha\), it lies in one of the already proved sectors: crossing support; a one-sided seam; the genuine double-exterior seam; or the seam cases \(V\in\{0,1\}\) with nonnegative \(r\).  Hence again \(P_W\ge0\).

The monotonicity inequalities along the reduction give \(P_Z\ge P_W\ge0\).  Reciprocal complement gives \(P_{Y'}\ge0\); because \(Y'\) is a predecessor of \(Y\), Lemma~\ref{lem:finite-difference-monotonicity} gives \(P_Y\ge P_{Y'}\ge0\); reciprocating once more gives \(P_X\ge0\), a contradiction.
\end{proof}

\begin{theorem}[One-sided exterior theorem]
\label{thm:core}
Assume
\begin{equation*}
a>\alpha,
\qquad c\le\delta.
\end{equation*}
If \(P(t,1,1)\ge0\) for every \(t>0\), then \(P(x,y,z)\ge0\) for all positive \(x,y,z\).
\end{theorem}

\begin{proof}
Reduce to a minimal collision-positive state.  Crossing support is Theorem~\ref{thm:crossing-support}; the seam is Theorem~\ref{thm:genuine-seam-one-sided}; and every state with \(\chi\le0\), including the boundary slices \(c=0\) and \(c=\delta\), is Proposition~\ref{prop:partii-negative-r}.  These cases are exhaustive.
\end{proof}

By Lemma~\ref{lem:reciprocal}, the reciprocal one-sided exterior region follows immediately from Theorem~\ref{thm:core}.
\section{Part IV: Negative-\texorpdfstring{\(r\)}{r} Completion}
\label{sec:negative-r}

\begin{theorem}[Completion of Part IV]
\label{thm:negative-r}
Assume
\begin{equation*}
a>\alpha,
\qquad c>\delta.
\end{equation*}
If \(P(t,1,1)\ge0\) for every \(t>0\), then \(P(x,y,z)\ge0\) for all positive \(x,y,z\).
\end{theorem}

\begin{proof}
Apply Minimal-State Reduction and let \(W\) be a terminal minimal collision-positive state.  If \(\chi_W\ge2\), use Theorem~\ref{thm:crossing-support}.  If \(\chi_W=1\) and \(V_W\ge2\), use Corollary~\ref{thm:double-cross-seam}.  If \(\chi_W=1\), \(V_W\in\{0,1\}\), and \(r_W\ge0\), use Theorem~\ref{thm:seam-v01-positive-r}.  Every remaining possibility is residual in the sense of \eqref{eq:residual-class-simplified}, so Proposition~\ref{prop:partii-negative-r} applies.  Hence \(P_W\ge0\), and the monotonicity inequalities give \(P\ge P_W\ge0\).
\end{proof}

\section{Completion of the Main Collision Criterion}
\label{sec:global-assembly}

Necessity in Theorem~\ref{thm:main} is immediate from the collision specialization.  For sufficiency, Part I is the midpoint theorem.  Part II is Theorem~\ref{thm:core}, Part III follows by reciprocal complement, and Part IV is Theorem~\ref{thm:negative-r}.  These four regions are exhaustive, so Theorem~\ref{thm:main} follows.

The simplified proof uses only: midpoint AM--GM/Muirhead; the monotonicity inequalities in the three predecessor directions and Minimal-State Reduction; the ordered-minor crossing/seam theorem; the common-middle lifting lemma; the multiplicative-kernel channel argument for \(V\in\{0,1\}\); Lemma~\ref{lem:negative-terminal-seam} for the negative-offset seam; and reciprocal maximal-depth descent.  In particular, the former weighted-tail recursion, Bernstein control arrays, neutral-boundary calculation, eight-ray legality cone, and the 117--780 term exact coefficient certificates are not logical dependencies of the theorem.

\section{Applications and a Refinement of Schur's Inequality}

The collision criterion has several immediate consequences.  First, clearing denominators extends the result to rational exponent vectors.  More importantly, the criterion converts questions about possible strengthenings of a three-variable symmetric inequality into one-variable questions on the collision locus.  We illustrate this principle with the classical Schur inequality.  We first show that the usual Schur chain is rigid if only one of its three exponent vectors is moved in a Muirhead direction that would make the inequality stronger.  We then show that this rigidity does not preclude a genuine strengthening outside that comparison direction: for cubic Schur, a second collision-positive midpoint is incomparable with the classical midpoint.  Thus the collision criterion both identifies a precise rigidity of Schur's inequality and produces a refinement lying outside the classical Muirhead order.

\subsection{Rational exponent vectors}

\begin{corollary}[Rational-exponent extension]\label{cor:rational}
The criterion of Theorem~\ref{thm:main} also holds for nonnegative rational exponent partitions.
\end{corollary}

\begin{proof}
Choose a positive integer \(m\) such that
\begin{equation*}
m\lambda,
\qquad
m\gamma,
\qquad
m\mu
\end{equation*}
are integer partitions.  Collision positivity of the rational chain implies collision positivity of the integer chain because
\begin{equation*}
P_{m\lambda,m\gamma,m\mu}(s,1,1)
=P_{\lambda,\gamma,\mu}(s^m,1,1).
\end{equation*}
Apply the completed integer collision criterion and then substitute
\begin{equation*}
s_1=x^{1/m},
\qquad
s_2=y^{1/m},
\qquad
s_3=z^{1/m}.
\end{equation*}
\end{proof}

\subsection{Low-degree integer examples beyond Schur}

The collision criterion gives many integer inequalities that do not have the classical Schur form
\begin{equation*}
(r+2,0,0)\succ(r+1,1,0)\succ(r,1,1).
\end{equation*}
Three low-degree examples are
\begin{align*}
J_{(9,0,0)}+J_{(3,3,3)}
&\ge2J_{(7,1,1)},\\
J_{(10,0,0)}+J_{(4,3,3)}
&\ge2J_{(8,1,1)},\\
J_{(9,1,0)}+J_{(4,3,3)}
&\ge2J_{(7,2,1)}.
\end{align*}

For the first inequality,
\begin{equation*}
Q_9(t)
=t^9-2t^7+3t^3-4t+2
=(t-1)^2H_9(t),
\end{equation*}
where
\begin{equation*}
H_9(t)=t^7+2t^6+t^5-t^3-2t^2+2>0
\qquad(t>0),
\end{equation*}
as shown in Section~4.4.

For the second inequality,
\begin{equation*}
Q_{10}^{(1)}(t)
=t^{10}-2t^8+t^4+2t^3-4t+2
=(t-1)^2H_{10}(t),
\end{equation*}
with
\begin{equation*}
H_{10}(t)
=t^8+2t^7+t^6-t^4-2t^3-2t^2+2.
\end{equation*}
The Sturm sequence of \(H_{10}\) has
\begin{equation*}
V(0^+)=V(+\infty)=4,
\end{equation*}
so \(H_{10}\) has no positive zero.  Since
\begin{equation*}
H_{10}(1)=1,
\end{equation*}
we have
\begin{equation*}
H_{10}(t)>0
\qquad(t>0).
\end{equation*}

For the third inequality,
\begin{align*}
Q_{10}^{(2)}(t)
&=t^9-2t^7+t^4+2t^3-2t^2-t+1\\
&=(t-1)^2(t+1)^2(t^5-t+1).
\end{align*}
Moreover,
\begin{equation*}
t^5-t+1>0
\qquad(t>0),
\end{equation*}
because
\begin{equation*}
t^5-t+1\ge1
\qquad(t\ge1),
\end{equation*}
and
\begin{equation*}
t^5-t+1>1-t>0
\qquad(0<t<1).
\end{equation*}
The full collision criterion therefore proves all three inequalities.

\subsection{Rigidity of the classical Schur chain}

For \(r\ge1\), set
\begin{equation*}
\eta=(r+2,0,0),
\qquad
\gamma=(r+1,1,0),
\qquad
\mu=(r,1,1).
\end{equation*}
The classical Schur inequality is
\begin{equation*}
J_\eta+J_\mu\ge2J_\gamma.
\end{equation*}

\begin{proposition}[One-vector Muirhead rigidity]\label{prop:schur-rigidity}
Keeping any two of \(\eta,\gamma,\mu\) fixed, the third cannot be moved nontrivially in the Muirhead direction that makes the Schur inequality stronger while preserving a majorization chain.
\end{proposition}

\begin{proof}
First keep \(\gamma,\mu\) fixed and replace \(\eta\) by a less-majorized partition \(\eta'=(a,b,c)\) satisfying \(\eta'\succ\gamma\).  Since the first two parts of \(\gamma\) already sum to the full degree \(r+2\), majorization forces \(c=0\).  If \(b>0\), the half-collision trace tends to \(1-2=-1\) as \(t\to0^+\), which is impossible under collision positivity.  Hence \(b=0\) and \(\eta'=\eta\).

Next keep \(\eta,\mu\) fixed.  Any strictly more-majorized midpoint has the form
\begin{equation*}
\gamma_s=(r+1+s,1-s,0),
\qquad
0<s\le1.
\end{equation*}
Its half-collision trace is
\begin{equation*}
t^{r+2}+t^r+2t-2t^{r+1+s}-2t^{1-s}.
\end{equation*}
For \(0<s<1\), the negative term \(-2t^{1-s}\) dominates as \(t\to0^+\); for \(s=1\), the constant term is negative.  Thus the midpoint cannot move upward.

Finally keep \(\eta,\gamma\) fixed and replace \(\mu\) by a more balanced partition \(\mu'=(a,b,c)\).  For \(r>1\), majorization gives \(c\ge1\).  If \(c>1\), the collision trace is \(-2t+o(t)\); if \(c=1\) but \(\mu'\ne\mu\), then \(b>1\) and the trace is \(-t+o(t)\).  For \(r=1\), \(\mu=(1,1,1)\) is already completely balanced.  Therefore \(\mu\) cannot move downward.
\end{proof}

\subsection{An incomparable refinement of cubic Schur}

\begin{corollary}\label{cor:sun-schur}
For all \(x,y,z>0\),
\begin{equation*}
J_{(3,0,0)}+J_{(1,1,1)}
\ge
2\max\left\{
J_{(2,1,0)},
J_{(9/4,\,1/2,\,1/4)}
\right\}.
\end{equation*}
Consequently this inequality strictly refines the usual cubic Schur inequality.
\end{corollary}

\begin{proof}
The first midpoint gives the classical cubic Schur inequality.  Put
\begin{equation*}
\widetilde\gamma
=\left(\frac94,\frac12,\frac14\right).
\end{equation*}
Then
\begin{equation*}
(3,0,0)\succ\widetilde\gamma\succ(1,1,1),
\end{equation*}
while \(\widetilde\gamma\) and \((2,1,0)\) are incomparable because
\begin{equation*}
\frac94>2,
\qquad
\frac94+\frac12<3.
\end{equation*}
By Corollary~\ref{cor:rational}, it remains to verify the collision restriction.  After removing the common factor \(2\), this restriction is
\begin{equation*}
R(t)
=t^3+3t+2-2t^{9/4}-2t^{1/2}-2t^{1/4}.
\end{equation*}
Set \(u=t^{1/4}\).  Then
\begin{equation*}
R(t)=(u-1)^2S(u),
\end{equation*}
where
\begin{equation*}
S(u)
=u^{10}+2u^9+3u^8+2u^7+u^6-u^4-2u^3+2u+2.
\end{equation*}
If \(u\ge1\), then
\begin{equation*}
S(u)
=u^{10}+2u^9+3u^8+2(u^7-u^3)+(u^6-u^4)+2u+2>0.
\end{equation*}
If \(0<u\le1\), then
\begin{align*}
S(u)
&\ge2+2u-u^4-2u^3\\
&\ge2+2u-3u^3\\
&=1+(1-u)(3u^2+3u+1)>0.
\end{align*}
Thus \(R(t)\ge0\) for every \(t>0\), and Corollary~\ref{cor:rational} gives
\begin{equation*}
J_{(3,0,0)}+J_{(1,1,1)}
\ge2J_{(9/4,1/2,1/4)}.
\end{equation*}
Combining this with the classical Schur inequality proves the stated maximum inequality.  The refinement is strict because, along \((t,1,1)\),
\begin{equation*}
J_{(9/4,1/2,1/4)}(t,1,1)\sim2t^{9/4},
\qquad
J_{(2,1,0)}(t,1,1)\sim2t^2
\end{equation*}
as \(t\to\infty\).
\end{proof}

\end{document}